\documentclass{article}
\usepackage{amsmath}
\usepackage{amssymb}
\usepackage{amsbsy}
\usepackage{mathrsfs}
\usepackage[all]{xypic}

\newcommand{\ra}{\rightarrow}

\newcommand{\bb}[1]{\mathbb{#1}}
\newcommand{\prm}{\ell}
\newcommand{\cop}{\ell}

\newcommand{\Z}{\bb{Z}}
\newcommand{\Q}{\bb{Q}}
\newcommand{\R}{\bb{R}}
\newcommand{\C}{\bb{C}}

\newcommand{\Zprm}{\Z_\prm}
\newcommand{\Zcop}{\Z_\cop}

\newcommand{\ten}{\otimes}
\newcommand{\teno}[1]{\ten_{#1}}

\newcommand{\proo}[1]{\lim_{\leftarrow #1}}
\newcommand{\gal}[1]{\mathrm{Gal}(#1)}

\newcommand{\br}[1]{\bar{#1}} % {\overline{#1}}

\newcommand{\st}{^\times}
\newcommand{\ie}{i.e. }

\newcommand{\Det}{\textstyle{\det}} % {\mathrm{det}}

\newcommand{\setm}{\! \smallsetminus \!}

\newcommand{\spec}{\mathrm{Spec}}

\newcommand{\het}{H_{\mathrm{\acute{e}t}}}
\newcommand{\oh}{\mathcal{O}}
\newcommand{\class}{\mathrm{Cl}}

\newcommand{\eg}{eg }

\newcommand{\spc}{\mbox{ }}
\newcommand{\Mod}{\spc\mathrm{mod}\spc}

\newcommand{\con}{\subseteq}

\newcommand{\inv}{\tau}
\newcommand{\ann}{\mathrm{ann}}

\newcommand{\iJ}[1]{\mathcal{J}(#1)}
\newcommand{\hJ}[2]{\mathcal{J}_{#2}^#1}

\newcommand{\E}{\mathcal{E}}

\newcommand{\Inf}{\mathrm{Inf}}

\newcommand{\tors}{\mathrm{Tors}}

\newcommand{\sat}{\spc | \spc}

\newcommand{\stick}{\theta}
\newcommand{\halfst}{\tilde{\stick}}

\newcommand{\rk}{\mathrm{rk}}

\newcommand{\rous}{\mu}
\newcommand{\rou}[1]{\rous(#1)}

\newcommand{\fitt}{\mathrm{Fitt}}

\newcommand{\Cprm}{\mathbb{C}_\prm}

\newcommand{\teich}{\omega}

\newcommand{\ab}{^{\mathrm{ab}}}

\newcommand{\BC}[1]{\mathrm{\Omega}^{#1}}
\newcommand{\irrch}[1]{\mathrm{ch}(#1)}
\newcommand{\Res}{\mathrm{Res}}
\newcommand{\Ind}{\mathrm{Ind}}

\newcommand{\chgrp}[1]{\widehat{#1}}
\newcommand{\cycq}[1]{\Q^{(#1)}}

\newcommand{\pwr}[1]{[\![#1]\!]}
\newcommand{\tsc}{\delta}

\newcommand{\VR}[1]{\mathcal{R}^{#1}}

\newcommand{\lint}{r}

\newcommand{\ncJ}[1]{I(#1)}
\newcommand{\comm}[1]{[#1,#1]}
\newcommand{\gf}{F}
\newcommand{\ef}{E}
\newcommand{\cycJ}{\iJ{\ef_n/\Q,S}}

\newtheorem{theorem}{Theorem}[section]
\newtheorem{dummy}[theorem]{}
\newtheorem{prop}[theorem]{Proposition}
\newtheorem{lemma}[theorem]{Lemma}

\newtheorem{definition}[theorem]{Definition}
\newtheorem{conj}[theorem]{Conjecture}
\newtheorem{question}[theorem]{Question}

\numberwithin{equation}{section}

\newenvironment{remark}
{
\vspace{0.25cm}
\emph{Remark}.
}
{
\vspace{1cm}
}

\newenvironment{proof}{\emph{Proof}.}{\hspace{\stretch{1}}\rule{1.5ex}{1.5ex} \vspace{5 mm}}

\begin{document}

\title{Functoriality of the canonical fractional Galois ideal}

\author{Paul Buckingham \and Victor Snaith}

\date{}

\maketitle

\begin{abstract}
The fractional Galois ideal of [Victor P. Snaith, Stark's conjecture and new Stickelberger phenomena, Canad. J. Math. 58 (2) (2006) 419--448] is a conjectural improvement on the higher Stickelberger ideals defined at negative integers, and is expected to provide non-trivial annihilators for higher $K$-groups of rings of integers of number fields. In this article, we extend the definition of the fractional Galois ideal to arbitrary (possibly infinite and non-abelian) Galois extensions of number fields under the assumption of Stark's conjectures, and prove naturality properties under canonical changes of extension. We discuss applications of this to the construction of ideals in non-commutative Iwasawa algebras.
\end{abstract}

\section{Introduction}

Let $\ef/\gf$ be a Galois extension of number fields with Galois group $G$. In seeking annihilators in $\Z[G]$ of the $K$-groups $K_{2n}(\oh_{\ef,S})$ ($S$ a finite set of places of $\ef$ containing the infinite ones), Stickelberger elements have long been a source of interest. This began with the classical Stickelberger theorem, showing that for abelian extensions $\ef/\Q$, annihilators of $\tors(K_0(\oh_{\ef,S}))$ can be constructed from Stickelberger elements. Coates and Sinnott later conjectured in \cite{cs:stickel} that the analogous phenomenon would occur for higher $K$-groups. However, defined in terms of values of $L$-functions at negative integers, these elements do not provide all the annihilators, because of the prevalent vanishing of the $L$-function values.

This difficulty is hoped to be overcome by considering the ``fractional Galois ideal'' introduced by the second author in \cite{snaith:rel,snaith:stark} and defined in terms of \emph{leading coefficients} of $L$-functions at negative integers under the assumption of the higher Stark conjectures. A version more suitable for the case of $\tors(K_0(\oh_{\ef,S})) = \class(\oh_{\ef,S})$ was defined in \cite{buckingham:frac} by the first author. Evidence that the fractional Galois ideal annihilates the appropriate $K$-groups (resp. class-groups) can be found in \cite{snaith:stark} (resp. \cite{buckingham:frac}). In the first case, it is \'etale cohomology that is annihilated, but this is expected to give $K$-theory by the Lichtenbaum--Quillen conjecture (see \cite[Section 1]{snaith:stark} for details).

With a view to relating the fractional Galois ideal to characteristic ideals in Iwasawa theory, we would like to describe how it behaves in towers of number fields. That it exhibits naturality in certain changes of extension was observed in particular cases in \cite{buckingham:frac}, and part of the aim of this paper is to explain these phenomena generally. Passage to subextensions corresponding to quotients of Galois groups will be of particular interest in the situation of non-abelian extensions, because of the relatively recent emergence of non-commutative Iwasawa theory in, for example, \cite{cfksv:main,fk:noncomm}. Consequently, the aims of this paper are

\begin{tabular}{cp{10cm}}
(i) & to prove formal properties of the fractional Galois ideal with respect to changes of extension, in the commutative setting first (\S \ref{2.3} to \S \ref{2.9}) \\
(ii) & to extend the definition of the fractional Galois ideal to non-abelian Galois extensions (\S \ref{4.1}), having previously defined it only for abelian extensions \\
(iii) & to show that it behaves well under passing to subextensions in the non-commutative setting also (Proposition \ref{4.4}) \\
(iv) & to show that in order for the non-commutative fractional Galois ideals to annihilate the appropriate \'etale cohomology groups, it is sufficient that the commutative ones do (\S \ref{5.4}).
\end{tabular}

We will also provide an explicit example (in the commutative case) in \S \ref{commutative example} illustrating how a limit of fractional Galois ideals gives the Fitting ideal for an inverse limit $\class_\infty$ of $\prm$-parts of class-groups. This should make clear the importance of taking leading coefficients of $L$-functions rather than just values, since it will be the part of the fractional Galois ideal corresponding to $L$-functions with first-order vanishing at $0$ which provides the Fitting ideal for the plus-part of $\class_\infty$.

In \S \ref{sec iwasawa theory}, we will conclude with a discussion of how the constructions of this paper fit into non-commutative Iwasawa theory. In particular, under some assumptions which, compared with the many conjectures permeating this area, are relatively weak, we will be able to give a partial answer to a question of Ardakov--Brown in \cite{ab:ringtheoretic} on constructing ideals in Iwasawa algebras.

\subsection{Error in Proposition~\ref{2.10}} \label{error}

Since the acceptance of the paper, the authors were made aware of a problem in Proposition~\ref{2.10}. It has to do with the fact that the induction map on representations is an additive homomorphism of representation rings, while the functoriality of $L$-functions refers to multiplication. Contrary to our expectations at the time, we have not been able to resolve this issue. We thank Andreas Nickel for bringing this problem to our attention.

\section{Notation and the Stark conjectures} \label{notation}

In what follows, by a Galois representation of a number field $F$ we shall mean a continuous, finite-dimensional complex representation of the absolute Galois group of $F$, which amounts to saying that the representation factors through the Galois group ${\rm Gal}(\ef/\gf)$ of a finite Galois extension $\ef/\gf$. We begin with the Stark conjecture (at $s=0$) and its generalizations to $s= -1, -2, -3, \ldots $ which were introduced in \cite{gross:higherstark} and \cite{snaith:stark} independently. 

Let $\Sigma(E)$ denote the set of embeddings of $E$ into the  complex numbers. For $r = 0,  -1, -2, -3, \ldots$ set
\[ Y_{r}(E) = \prod_{\Sigma(E)} \ (2 \pi i )^{-r} {\mathbb Z}  =  {\rm Map}( \Sigma(E) , (2 \pi i )^{-r} {\mathbb Z} )  \]
endowed with the $G({\mathbb C}/{\mathbb R})$-action diagonally on $\Sigma(E)$ and on $(2 \pi i )^{-r} $. If $c_{0}$ denotes complex conjugation, the action of $c_{0}$ and $G$ commute so that
the fixed points of $Y_{r}(E)$ under $c_{0}$, denoted by $Y_{r}(E)^{+}$, form a $G$-module. It is easy to see that the rank of $ Y_{r}(E)^{+} $ is given by
\[ \rk_{{\mathbb Z}}( Y_{r}(E)^{+} ) = \left\{
\begin{array}{ll}
r_{2} & {\rm if } \   r  \  {\rm is \ odd}, \\
r_{1} + r_{2} & {\rm if } \   r \geq 0  \  {\rm is \ even}   .
\end{array} \right. \]
where $| \Sigma(E)|  = r_{1} + 2r_{2}$ and $r_{1}$ is the number of real embeddings of $E$.

\subsection{Stark regulators} \label{1.2}
 
We begin with a slight modification of the original Stark regulator \cite{tate:stark}. Now let $G$ denote the Galois group of an extension of number fields $\ef/\gf$. 
We extend the Dirichlet regulator homomorphism to the Laurent polynomials with coefficients in ${\cal O}_{E}$ to give an ${\mathbb R}[G]$-module isomorphism of the form
\[ R_{E}^{0} :  K_{1}({\cal O}_{E} \langle t^{\pm 1} \rangle ) \otimes {\mathbb R} = {\cal O}_{E} \langle t^{\pm 1} \rangle\st \otimes {\mathbb R} \stackrel{\cong}{\ra} Y_{0}(E)^{+} \otimes {\mathbb R} \cong  {\mathbb R}^{r_{1}+r_{2}}   \]
by the formulae, for $u \in {\cal O}_{E} \st$,
\[  R_{E}^{0}( u ) =  \sum_{\sigma \in   \Sigma(E)}  \  {\rm log}(|\sigma(u)|) \cdot  \sigma  \]
and
\[     R_{E}^{0}( t  ) =   \sum_{\sigma \in   \Sigma(E)}  \    \sigma .      \]
The existence of this isomorphism implies (see \cite[Section 12.1]{serre:linear} and \cite[p.26]{tate:stark}) that there exists at least one   ${\mathbb Q}[G]$-module isomorphism of the form
\[ f_{E }^{0} :  {\cal O}_{E} \langle t^{\pm 1} \rangle\st \otimes {\mathbb Q}
\stackrel{\cong}{\ra} Y_{0}(E )^{+} \otimes {\mathbb Q}  .  \]
For any choice of $ f_{E }^{0}$ Stark forms the composition
\[ R_{E }^{0} \cdot (f_{E }^{0})^{-1} :   Y_{0}(E )^{+} \otimes {\mathbb C}
\stackrel{\cong}{\ra} Y_{0}(E )^{+} \otimes {\mathbb C}  \]
which is an isomorphism of complex representations of $G$. Let $V$ be a finite-dimensional complex representation of $G$ whose contragredient is denoted by $V^{\vee}$. The Stark regulator is defined to be the exponential homomorphism $V \mapsto  R(V, f_{E }^{0})$, from representations to non-zero complex numbers, given by
\[  R(V, f_{E }^{0}) = \Det( (R_{E}^{0} \cdot  (f_{E }^{0})^{-1})_{*} \in {\rm Aut}_{{\bf C}}({\rm Hom}_{G}( V^{\vee} ,  Y_{0}(E )^{+} \otimes {\mathbb C}) ))  \]
where $ (R_{E }^{0} \cdot (f_{E }^{0})^{-1})_{*}$ is composition with $ R_{E }^{0} \cdot (f_{E }^{0})^{-1}$.

For $r =  -1, -2, -3, \ldots$ there is an isomorphism of the form \cite{quillen:ktheory}
\[    K_{1-2r}({\cal O}_{E} \langle t^{\pm 1} \rangle ) \otimes {\mathbb Q}  \cong 
 K_{1-2r}({\cal O}_{E}  ) \otimes {\mathbb Q}   \]
because $ K_{-2r}({\cal O}_{E}  )$ is finite. Therefore the Borel regulator homomorphism defines an  ${\mathbb R}[G]$-module isomorphism of the form
\[ R_{E}^{r} :  K_{1-2r}({\cal O}_{E} \langle t^{\pm 1} \rangle ) \otimes {\mathbb R} = K_{1-2r}({\cal O}_{E} ) \otimes {\mathbb R} \stackrel{\cong}{\ra} Y_{r}(E)^{+} \otimes {\mathbb R} . \]
Choose a  ${\mathbb Q}[G]$-module isomorphism of the form
\[ f_{E }^{r} :  K_{1-2r}({\cal O}_{E} \langle t^{\pm 1} \rangle ) \otimes {\mathbb Q} \stackrel{\cong}{\ra} Y_{r}(E)^{+} \otimes {\mathbb Q} \]
and form the analogous Stark regulator, $(V \mapsto  R(V, f_{E }^{r})) $, from representations to non-zero complex numbers given by
\[  R(V, f_{E }^{r}) = \Det( (R_{E}^{r} \cdot  (f_{E }^{r})^{-1})_{*} \in {\rm Aut}_{{\bf C}}({\rm Hom}_{G}( V^{\vee} ,  Y_{r}(E )^{+} \otimes {\mathbb C}) )) .   \]

\subsection{Stark's conjectures} \label{stark conjectures}

Let $R(G)$ denote the complex representation ring of the finite group $G$; that is, $R(G) = K_{0}({\mathbb C}[G])$. Since $V$ determines a Galois representation of $F$, we have a non-zero complex number $L_{F}^{*}(r, V )$ given by the leading coefficient of the Taylor series at $s=r$ of the Artin $L$-function associated to $V$  (\cite{martinet:lfunctions}, \cite[p.23]{tate:stark}). 

We may modify $ R(V, f_{E }^{r})$ to give another exponential homomorphism
\[  {\cal R}_{f_{E }^{r}}  \in {\rm Hom}( R( G) ,{\mathbb C}\st)   \]
defined by
\[ {\cal R}_{f_{E }^{r}}(V) = \frac{R(V,f_{E }^{r})}{L_{F}^{*}( r , V) }.  \]
Let $\overline{{\mathbb Q}}$ denote the algebraic closure of the rationals in the complex numbers and let $\Omega_{{\mathbb Q}}$ denote the absolute Galois group of the rationals, which acts continuously on $R( G) $ and $\overline{{\mathbb Q}}\st$. The Stark conjecture asserts that for each $r= 0, -1, -2, -3, \ldots$
\[  {\cal R}_{f_{E }^{r}} \in {\rm Hom}_{\Omega_{{\mathbb Q}}}( R( G) , \overline{{\mathbb Q}}\st)  \subseteq {\rm Hom}( R( G ) ,{\mathbb C}\st) . \]
In other words, ${\cal R}_{f_{E }^{r}}(V)$ is an algebraic number for each $V$ and for all $z \in \Omega_{{\mathbb Q}}$ we have $z({\cal R}_{f_{E }^{r}}(V)) = {\cal R}_{f_{E }^{r}}(z(V))$. Since any two choices of $f_{E}^{r}$ differ by multiplication by a ${\mathbb Q}[G]$-automorphism, the truth of the conjecture is independent of the choice of $f_{E }^{r}$ (\cite{tate:stark} pp.28-30).  

When $s=0$ the conjecture which we have just formulated apparently differs from the classical Stark conjecture of \cite{tate:stark}, therefore we shall pause to show that the two conjectures are equivalent. For the classical Stark conjecture one replaces $Y_{0}(E )^{+}$ by $X_{0}(E)^{+}$ where $X_{0}(E)$ is the kernel of the augmentation homomorphism $Y_{0}(E ) \ra  {\mathbb Z}$, which adds together all the coordinates. The Dirichlet regulator  gives an ${\mathbb R}[G]$-module isomorphism
\[ \tilde{R}_{E}^{0} :  {\cal O}_{E}\st  \otimes {\mathbb R} \stackrel{\cong}{\ra} X_{0}(E)^{+} \otimes {\mathbb R}  \]
and choosing a ${\mathbb Q}[G]$-module isomorphism
\[ \tilde{f}_{E}^{0} :  {\cal O}_{E}\st  \otimes {\mathbb Q} \stackrel{\cong}{\ra} X_{0}(E)^{+} \otimes {\mathbb Q}  \]
we may form
\[ \tilde{R}_{E }^{0} \cdot (\tilde{f}_{E }^{0})^{-1} :   X_{0}(E )^{+} \otimes {\mathbb C} \stackrel{\cong}{\ra} X_{0}(E )^{+} \otimes {\mathbb C}  .  \]
Taking its Stark determinant we obtain $\tilde{R}(V, \tilde{f}_{E }^{0})$ and finally
\[   \tilde{{\cal R}}_{ \tilde{f}_{E}^{0}}(V) = \frac{\tilde{R}(V, \tilde{f}_{E }^{0})}{L_{F}^{*}(0 , V) }.  \]

\begin{prop}
In \S \ref{stark conjectures}
\[   {\cal R}_{ f_{E}^{0}} \in {\rm Hom}_{\Omega_{{\mathbb Q}}}( R( G) , \overline{{\mathbb Q}}\st)  \subseteq {\rm Hom}( R( G ) ,{\mathbb C}\st)      \]
if and only if
\[  \tilde{{\cal R}}_{ \tilde{f}_{E}^{0}}  \in {\rm Hom}_{\Omega_{{\mathbb Q}}}( R( G) , \overline{{\mathbb Q}}\st)  \subseteq {\rm Hom}( R( G ) ,{\mathbb C}\st)  \]
independently of the choice of $f_{E}^{0}$ or $\tilde{f}_{E}^{0}$.
\end{prop}

\begin{proof}
Given any ${\mathbb Q}[G]$-isomorphism $\tilde{f}_{E}^{0}$ we may fill in the following commutative diagram by ${\mathbb Q}[G]$-isomorphisms  $f_{E}^{0}$ and $\overline{f}_{E}^{0}$. Conversely, given any ${\mathbb Q}[G]$-isomorphisms $f_{E}^{0}$ and $\overline{f}_{E}^{0}$ we may fill in the diagram
with a ${\mathbb Q}[G]$-isomorphism $\tilde{f}_{E}^{0}$.  
\[ \xymatrix{
\oh_E\st \teno{\Z} \Q \ar[r] \ar[d]^{\tilde{f}_E^0} & \oh_E[t^{\pm 1}]\st \teno{\Z} \Q \ar[r] \ar[d]^{f_E^0} & \Q \ar[d]^{\br{f}_E^0} \\
X_0(E)^+ \teno{\Z} \Q \ar[r] & Y_0(E)^+ \teno{\Z} \Q \ar[r] & \Q
} \]

Similarly there is a commutative diagram in which the vertical arrows are reversed, ${\mathbb Q}$ is replaced by ${\mathbb R}$ and $\tilde{f}_{E}$,  $f_{E}$ and $\overline{f}_{E}$ by $\tilde{R}_{E}^{0} $, $R_{E}^{0}$ and $\overline{R}_{E }^{0}$, respectively.  Furthermore $\overline{R}_{E }^{0}$ is multiplication by a rational number. The result now follows from the multiplicativity of the determinant in short exact sequences.
\end{proof}

We shall be particularly interested in the case when $G $ is abelian, in which case the following observation is important. Let $\chgrp{G} = {\rm Hom}(G,  \overline{{\mathbb Q}}\st)$ denote the set of characters on $G$ and let ${\mathbb Q}(\chi)$ denote the field generated by the character values of a representation $\chi$. We may identify ${\rm Hom}_{\Omega_{{\mathbb Q}}}( R(G) , \overline{{\mathbb Q}})  $with the ring ${\rm Map}_{\Omega_{{\mathbb Q}}}(\chgrp{G},  \overline{{\mathbb Q}})$.

\begin{prop} \label{1.5}
Let $G$ be a finite abelian group. Then there exists an isomorphism of rings
\[ \lambda_{G} : {\rm Map}_{\Omega_{{\mathbb Q}}}(\chgrp{G},  \overline{{\mathbb Q}})={\rm Hom}_{\Omega_{{\mathbb Q}}}( R(G) , \overline{{\mathbb Q}}) \stackrel{\cong}{\ra}  {\mathbb Q}[G]  \]
given by
\[  \lambda_{G}(h) =  \sum_{\chi \in \chgrp{G}} \  h(\chi) e_{\chi} \]
where
\[  e_{\chi} = |G|^{-1} \sum_{g \in G} \ \chi(g) g^{-1} \in {\mathbb Q}(\chi)[G]. \]
In particular there is an isomorphism of unit groups
\[ \lambda_{G} : {\rm Hom}_{\Omega_{{\mathbb Q}}}( R(G) , \overline{{\mathbb Q}}\st) \stackrel{\cong}{\ra}  {\mathbb Q}[G]\st .   \]
\end{prop}

\begin{proof}
There is a well-known isomorphism of rings (\cite{lang:algebra2nd} p.648)
\[  \psi : \overline{{\mathbb Q}}[G]  \ra  \prod_{  \chi \in \chgrp{G} } \overline{{\mathbb Q}}  = {\rm Map}( \chgrp{G} ,  \overline{{\mathbb Q}})  \]
given by $\psi( \sum_{g \in G} \lambda_{g} g)( \chi) =  \sum_{g \in G} \lambda_{g} \chi(g)$. If $\Omega_{{\mathbb Q}}$ acts on $\overline{{\mathbb Q}}$ and $\chgrp{G}$ in the canonical manner, then $\psi$ is Galois equivariant and induces an isomorphism of $\Omega_{{\mathbb Q}}$-fixed points of the form
\[  {\mathbb Q}[G]  =  (\overline{{\mathbb Q}}[G] )^{\Omega_{{\mathbb Q}}} \cong  
 {\rm Map}_{\Omega_{{\mathbb Q}}}( \chgrp{G} ,  \overline{{\mathbb Q}}) \cong {\rm Hom}_{\Omega_{{\mathbb Q}}}( R(G) , \overline{{\mathbb Q}})   .  \]
It is straightforward to verify that this isomorphism is the inverse of $\lambda_{G}$.
\end{proof}

\section{The canonical fractional Galois ideal ${\cal J}_{\ef/\gf}^{r}$ in the abelian case}

\subsection{Definition of ${\cal J}_{\ef/\gf}^r$} \label{2.1}

In this section we recall the canonical fractional Galois ideal introduced in \cite{snaith:stark} (see also \cite{buckingham:frac}, \cite{snaith:equiv} and  \cite{snaith:rel}). In \cite{snaith:stark} this was denoted merely by  ${\cal J}_{E}^{r}$ but in this paper we shall need to keep track of the base field.

As in \S \ref{stark conjectures}, let $\ef/\gf$ be a Galois extension of number fields. Throughout this section we shall assume that the Stark conjecture of \S \ref{stark conjectures} is true for all $\ef/\gf$ and  that $G = {\rm Gal}(\ef/\gf)$ is abelian. Therefore, by Proposition \ref{1.5}, for each $r = 0, -1, -2, -3, \ldots$ 
we have an element
\[  {\cal R}_{f_{E }^{r}}  \in {\rm Hom}_{\Omega_{{\mathbb Q}}}( R( G ) , \overline{{\mathbb Q}}\st)  \cong  {\mathbb Q}[G]\st  \]
which depends upon the choice of a ${\mathbb Q}[G]$-isomorphism $f_{E}^{r}$ in \S \ref{stark conjectures}.

Let $ \alpha \in  {\rm End}_{{\mathbb Q}[G]}( Y_{r}(E)^{+} \otimes {\mathbb Q} )$ and extend this by the identity on the $(-1)$-eigenspace of complex conjugation $Y_{r}(E)^{-}   \otimes {\mathbb Q} $ to give
\[    \alpha  \oplus   1  \in   {\rm End}_{{\mathbb Q}[G]}( Y_{r}(E) \otimes {\mathbb Q} )   .   \]
Since $Y_{r}(E) \otimes {\mathbb Q} $ is free over ${\mathbb Q}[G]$, we may form the determinant
\[    {\rm det}_{{\mathbb Q}[G] }( \alpha \oplus 1) \in {\mathbb Q}[G]   .   \]
In terms of the isomorphism of Proposition \ref{1.5}, ${\rm det}_{{\mathbb Q}[G] }( \alpha \oplus 1)$ corresponds to the function which sends $\chi \in \chgrp{G}$ to the determinant of the endomorphism of  $e_{\chi} Y_{r}(E) \otimes \overline{{\mathbb Q}}$ induced by $\alpha \oplus 1$.

Following \cite[Section 4.2]{snaith:stark} (see also \cite{snaith:rel,snaith:equiv}), define ${\cal I}_{f_{E}^{r}}$ to be the (finitely generated) ${\mathbb Z}[1/2][G]$-submodule of ${\mathbb Q}[ G ]$ generated by all the elements \linebreak ${\rm det}_{{\mathbb Q}[G] }( \alpha \oplus 1) $ satisfying the integrality condition
\[  \alpha \cdot   f_{E}^{r}( K_{1-2r}({\cal O}_{E}[t^{\pm 1}]))      \subseteq Y_{r}(E) .  \]

Define ${\cal J}_{\ef/\gf}^{r}$ to be the finitely generated ${\mathbb Z}[1/2][G]$-submodule of ${\mathbb Q}[ G ]$ given by
\[ {\cal J}_{\ef/\gf}^{r} =    {\cal I}_{ f_{E}^{r}} \cdot  \inv(   {\cal R}_{f_{E}^{r}}^{-1})  \]
where $\inv$ is the  automorphism of the group-ring induced by sending each $g \in G$ to its inverse.

\begin{prop}{(\cite[Prop.4.5]{snaith:stark})}
Let $\ef/\gf$ be a Galois extension of number fields with abelian Galois group $G$. Then, assuming that the Stark conjecture of \S \ref{stark conjectures} holds for $\ef/\gf$ for $r= 0, -1, -2, -3, \ldots$, the finitely generated ${\mathbb Z}[1/2][G]$-submodule ${\cal J}_{\ef/\gf}^{r}$ of ${\mathbb Q}[ G ]$ just defined is independent of the choice of $f_{E}^{r}$.
\end{prop}

\subsection{Naturality examples} \label{nat examples}

Given an extension $\ef/\gf$ of number fields satisfying the Stark conjecture at $s = 0$ and a finite set of places $S$ of $\gf$ containing the infinite places, let $\iJ{\ef/\gf,S}$ denote the fractional Galois ideal as defined in \cite{buckingham:frac}, a slight modification of the one just defined so that we can take into account finite places. Let us consider the following situation: $\prm$ is an odd prime, $\ef_n = \Q(\zeta_{\prm^{n+1}})$ for a primitive $\prm^{n+1}$th root of unity $\zeta_{\prm^{n+1}}$ ($n \geq 0$), and $S = \{\infty,\prm\}$. The descriptions below of $\iJ{\ef_n/\Q,S}$ and $\iJ{\ef_n^+/\Q,S}$ are provided in \cite[Section 4]{buckingham:frac}:
\begin{eqnarray}
\iJ{\ef_n/\Q,S} &=& \frac{1}{2} e_+ \ann_{\Z[G_n]}(\oh_{\ef_n^+,S}\st/\E_n^+) \oplus \Z[G_n]\stick_{\ef_n/\Q,S} \label{full j desc} \label{recap full j desc} \\
\iJ{\ef_n^+/\Q,S} &=& \frac{1}{2} \ann_{\Z[G_n^+]}(\oh_{\ef_n^+,S}\st/\E_n^+) \label{real j desc}
\end{eqnarray}
where $G_n = \gal{\ef_n/\Q}$, $G_n^+ = \gal{\ef_n^+/\Q}$, $\E_n^+$ is the $\Z[G_n^+]$-submodule of $\oh_{\ef_n^+,S}\st$ generated by $-1$ and $(1 - \zeta_{\prm^{n+1}})(1 - \zeta_{\prm^{n+1}}^{-1})$, and $\stick_{\ef_n/\Q,S}$ is the Stickelberger element at $s = 0$. Also, $e_+ = \frac{1}{2}(1+c)$ is the plus-idempotent for complex conjugation $c \in G_n$.

It is immediate from these descriptions that the natural maps $\Q[G_n] \ra \Q[G_n^+]$, $\Q[G_n] \ra \Q[G_{n-1}]$ and $\Q[G_n^+] \ra \Q[G_{n-1}^+]$ give rise to a commutative diagram
\begin{equation} \label{quot examples}
\xymatrix{
\iJ{\ef_n/\Q,S} \ar[r] \ar[d] & \iJ{\ef_n^+/\Q,S} \ar[d] \\
\iJ{\ef_{n-1}/\Q,S} \ar[r] & \iJ{\ef_{n-1}^+/\Q,S} .
}
\end{equation}
($\oh_{\ef_{n-1}^+,S}\st/\E_{n-1}^+$ embeds into $\oh_{\ef_n^+,S}\st/\E_n^+$, and Stickelberger elements are well known (e.g. \cite{hayes:nonabstick}) to map to each other in this way.)

Now suppose that $\prm \equiv 3 \Mod 4$, so that $\ef_n$ contains the imaginary quadratic field $\gf = \Q(\sqrt{-\prm})$. Again, letting $S_\gf$ consist of the infinite place of $\gf$ and the unique place above $\prm$, $\iJ{\ef_n/\gf,S_\gf}$ has a simple description. Indeed, if $H_n = \gal{\ef_n/\gf}$, then
\begin{equation} \label{lifted j desc}
\iJ{\ef_n/\gf,S_\gf} = \frac{1}{\rous_n} \ann_{\Z[H_n]}(\oh_{\ef_n,S}\st/\E_n)
\end{equation}
where $\E_n$ is generated over $\Z[H_n]$ by $\zeta_{\prm^{n+1}}$ and $(1 - \zeta_{\prm^{n+1}})^{\rous_n \halfst_n}$. Here, $\rous_n = |\rou{\ef_n}|$ and $\halfst_n = \sum_{\sigma \in H_n} \zeta_{\ef_n/\Q,S}(0,\sigma^{-1}) \sigma \in \Q[H_n]$, a sort of ``half Stickelberger element'' obtained by keeping only those terms corresponding to elements in the index two subgroup $H_n$ of $G_n$. (Note that $\rous_n \halfst_n \in \Z[H_n]$.) Comparing (\ref{real j desc}) and (\ref{lifted j desc}), we see without too much difficulty that

\begin{prop} \label{nat irrats example}
The isomorphism $\Phi_n : \Q[H_n] \ra \Q[G_n^+]$ identifies $\iJ{\ef_n/\gf,S_\gf}$ with $2\Phi_n(\halfst_n) \iJ{\ef_n^+/\Q,S}$.
\end{prop}

We now explain the above phenomena by proving some general relationships between the ${\cal J}_{\ef/\gf}^\lint$ under natural changes of extension.

\subsection{Behaviour under quotient maps ${\rm Gal}(L/F) \ra  {\rm Gal}(K/F)$} \label{2.3}

Suppose that $F \subseteq K \subseteq L$ is a tower of number fields with $L/F$ abelian. The inclusion of $K$ into $L$ induces a homomorphism
\[   K_{1-2r}({\cal O}_{K}[t^{\pm 1}])  \ra    K_{1-2r}({\cal O}_{L}[t^{\pm 1}])  .   \]
When $r=0$
\[   \frac{ K_{1}({\cal O}_{K}[t^{\pm 1}])}{{\rm Torsion}}   \cong  {\cal O}_{K}\st/( \mu(K))  \oplus  {\mathbb Z} \langle t \rangle  \]
maps injectively to the Galois invariants of $ {\cal O}_{L}\st/( \mu(L))  \oplus {\mathbb Z} \langle t \rangle $ sending $t$ to itself.  For strictly negative $r$,
\[     \frac{ K_{1-2r}({\cal O}_{K}[t^{\pm 1}])}{{\rm Torsion}}   \cong \frac{ K_{1-2r}({\cal O}_{K})}{{\rm Torsion}}   \]
embeds into the ${\rm Gal}(L/K)$-invariants of $  \frac{ K_{1-2r}({\cal O}_{L}[t^{\pm 1}])}{{\rm Torsion}}  $. There is a homomorphism $Y_{r}(K)  \ra  Y_{r}(L)$ which sends $n_{\sigma} \cdot \sigma$ to $n_{\sigma} \cdot ( \sum_{ (   \sigma'  \ | \ F) = \sigma}  \  \sigma' )$ which is an isomorphism onto the ${\rm Gal}(L/K)$-invariants $Y_{r}(L)^{{\rm Gal}(L/K)}$. For $ r = 0, -1, -2, -3, \ldots$ there is a commutative diagram of regulators in \S\ref{1.2}
\[ \xymatrix{
K_{1-2r}(\oh_K[t^{\pm 1}]) \teno{\Z} \R \ar[r]^-{R_K^r} \ar[d] & Y_r(K)^+ \teno{\Z} \R \ar[d] \\
K_{1-2r}(\oh_L[t^{\pm 1}]) \teno{\Z} \R \ar[r]^-{R_L^r} & Y_r(L)^+ \teno{\Z} \R
} \]
We may choose $f_{K}^{r}$ and $f_{L}^{r}$ as in \S\ref{1.2} to make the corresponding diagram of ${\mathbb Q}$-vector spaces commute
\begin{equation} \label{f chosen to commute}
\xymatrix{
K_{1-2r}(\oh_K[t^{\pm 1}]) \teno{\Z} \Q \ar[r]^-{f_K^r} \ar[d] & Y_r(K)^+ \teno{\Z} \Q \ar[d] \\
K_{1-2r}(\oh_L[t^{\pm 1}]) \teno{\Z} \Q \ar[r]^-{f_L^r} & Y_r(L)^+ \teno{\Z} \Q
}
\end{equation}

Let $V$ be a one-dimensional complex representation of ${\rm Gal}(K/F)$ and let $W = \Inf_{{\rm Gal}(K/F)}^{{\rm Gal}(L/F)}(V)$ denote the inflation of $V$. Then
\[  \begin{array}{l}
{\rm Hom}_{{\rm Gal}(L/F)}(W^{\vee} ,  Y_{r}(L)^{+}  \otimes {\mathbb C})  \\
=  {\rm Hom}_{{\rm Gal}(L/F)}(W^{\vee} ,  (  Y_{r}(L)^{{\rm Gal}(L/K)})^{+}  \otimes {\mathbb C})  \\
=  {\rm Hom}_{{\rm Gal}(K/F)}(V^{\vee} ,  Y_{r}(K)^{+}  \otimes {\mathbb C})
\end{array}  \]
and these isomorphisms transport $( R_{L}^{r} \cdot (f_{L}^{r})^{-1})_{*}$ into $( R_{K}^{r} \cdot (f_{K}^{r})^{-1})_{*}$ by virtue of the above commutative diagrams. Furthermore, since the Artin $L$-function is invariant under inflation, $L_{F}^{*}(r, V) = L_{F}^{*}(r , W)$. On the other hand, the inflation homomorphism
\[   \Inf_{{\rm Gal}(K/F)}^{{\rm Gal}(L/F)} :  R( {\rm Gal}(K/F) )  \ra  R({\rm Gal}(L/F))    \]
induces the canonical quotient map
\[   \pi_{L/K} :  {\mathbb Q}[{\rm Gal}(L/F)]\st  \ra     {\mathbb Q}[ {\rm Gal}(K/F) ]\st    \]
via the isomorphism of Proposition \ref{1.5}. Hence
\[    \pi_{L/K} (    {\cal R}_{f_{L}^{r}}  )  =   {\cal R}_{f_{K}^{r}}  . \]

Let $ \alpha \in  {\rm End}_{{\mathbb Q}[{\rm Gal}(L/F)]}( Y_{r}(L)^{+} \otimes {\mathbb Q} )$ satisfy the integrality condition of \S\ref{2.1}
\[  \alpha \cdot   f_{L}^{r}( K_{1-2r}({\cal O}_{L}[t^{\pm 1}]))      \subseteq Y_{r}(L) .  \]
Extend this by the identity on the $(-1)$-eigenspace of complex conjugation $Y_{r}(L)^{-}   \otimes {\mathbb Q} $ to give
\[    \alpha  \oplus   1  \in   {\rm End}_{{\mathbb Q}[{\rm Gal}(L/F)]}( Y_{r}(L) \otimes {\mathbb Q} )   .   \]
The endomorphism $\alpha$ commutes with the action by ${\rm Gal}(L/K)$ so there is $\hat{\alpha}  \in   {\rm End}_{{\mathbb Q}[{\rm Gal}(K/F)]}( Y_{r}(K)^{+} \otimes {\mathbb Q} )$ making the following diagram commute
\[ \xymatrix{
Y_r(K)^+ \teno{\Z} \Q \ar[r]^{\hat{\alpha}} \ar[d] & Y_r(K)^+ \teno{\Z} \Q \ar[d] \\
Y_r(L)^+ \teno{\Z} \Q \ar[r]^\alpha & Y_r(L)^+ \teno{\Z} \Q .
} \]
Therefore $\hat{\alpha}$ satisfies the integrality condition of \S\ref{2.1}
\[  \hat{\alpha} \cdot   f_{K}^{r}( K_{1-2r}({\cal O}_{K}[t^{\pm 1}]))      \subseteq Y_{r}(K) .  \]
We may choose a ${\mathbb Z}[1/2][{\rm Gal}(K/F)]$ basis for $Y_{r}(K) \otimes {\mathbb Z}[1/2]$ consisting of embeddings $\sigma_{i} : K \ra  {\mathbb C}$ for $1 \leq i \leq m$. Let $ \sigma'_{i}$ be an embedding of $L$ which extends $\sigma_{i}$ for $1 \leq i \leq m$. Then a ${\mathbb Z}[1/2][{\rm Gal}(L/F)]$ basis for $Y_{r}(L) \otimes {\mathbb Z}[1/2]$ is given by $\{   \sigma'_{1} ,  \sigma'_{2} , \ldots ,  \sigma'_{m} \}$. The embedding of $Y_{r}(K) $ into $Y_{r}(L)$ is given by  $\sigma_{i}   \mapsto  \sum_{g \in {\rm Gal}(L/K)}  \  g(\sigma'_{i} ) $ which implies that the $m \times m$ matrix for $\hat{\alpha}$ with respect to the ${\mathbb Z}[1/2][{\rm Gal}(K/F)]$ basis of $\sigma_{i}$'s is the image of the  $m \times m$ matrix for $\alpha$ with respect to the ${\mathbb Z}[1/2][{\rm Gal}(L/F)]$ basis of $\sigma'_{i}$'s under the canonical surjection
\[    {\mathbb Q}[{\rm Gal}(L/F)]  \ra    {\mathbb Q}[{\rm Gal}(K/F)]   .  \]

This discussion has established the following result.

\begin{prop} \label{2.4}
Suppose that $F \subseteq K \subseteq L$ is a tower of number fields with $L/F$ abelian. Then, in the notation of \S\ref{2.1},  the canonical surjection
\[    \pi_{L/K} :  {\mathbb Q}[{\rm Gal}(L/F)]  \ra    {\mathbb Q}[{\rm Gal}(K/F)]     \]
satisfies
\[    \pi_{L/K}(  {\cal J}_{L/F}^{r} )  \subseteq  {\cal J}_{K/F}^{r} .  \]
\end{prop}

Proposition \ref{2.4} explains the existence of the maps in (\ref{quot examples}).

\subsection{Behaviour under inclusion maps ${\rm Gal}(L/K) \ra  {\rm Gal}(L/F)$} \label{2.5}

As in \S\ref{2.4}, suppose that $F \subseteq K \subseteq L$ is a tower of number fields with $L/F$ abelian. The inclusion of ${\rm Gal}(L/K)$ into ${\rm Gal}(L/F)$ induces an inclusion of group-rings ${\mathbb Q}[{\rm Gal}(L/K)]$ into  ${\mathbb Q}[{\rm Gal}(L/F)]$. In terms of the isomorphism of Proposition \ref{1.5}, as is easily seen by the formula,  this homomorphism is induced by the restriction of representations
\[   \Res_{{\rm Gal}(L/K)}^{{\rm Gal}(L/F)} :   R({\rm Gal}(L/F))  \ra  R({\rm Gal}(L/K)) . \]

If $V$ is a complex representation of ${\rm Gal}(L/F)$ then
\begin{eqnarray*}
{\cal R}_{f_{L }^{r}}(  \Res_{{\rm Gal}(L/K)}^{{\rm Gal}(L/F)}(V)) &=& \frac{R(  \Res_{{\rm Gal}(L/K)}^{{\rm Gal}(L/F)}(V),f_{L}^{r})}{L_{K}^{*}( r ,  \Res_{{\rm Gal}(L/K)}^{{\rm Gal}(L/F)}( V)) }   \\
&=&   \frac{R(  \Res_{{\rm Gal}(L/K)}^{{\rm Gal}(L/F)}(V),f_{L}^{r})}{L_{F}^{*}( r , \Ind_{{\rm Gal}(L/K)}^{{\rm Gal}(L/F)}(\Res_{{\rm Gal}(L/K)}^{{\rm Gal}(L/F)}( V))) }   \\
&=&    \frac{R(  \Res_{{\rm Gal}(L/K)}^{{\rm Gal}(L/F)}(V),f_{L}^{r})}{L_{F}^{*}( r ,  V \otimes \Ind_{{\rm Gal}(L/K)}^{{\rm Gal}(L/F)}(1)) }  .
\end{eqnarray*}
If $W_{i}  \in  \chgrp{\rm Gal}(L/F)$ for $1 \leq i \leq [K : F]$  is the set of  one-dimensional representations which restrict to the trivial representation on ${\rm Gal}(L/K)$ then \linebreak $\Ind_{{\rm Gal}(L/K)}^{{\rm Gal}(L/F)}(1))  =  \oplus_{i} \  W_{i}$. By Frobenius reciprocity
\[   \begin{array}{l}
{\rm Hom}_{{\rm Gal}(L/K)}(   \Res_{{\rm Gal}(L/K)}^{{\rm Gal}(L/F)}(V)^{\vee} ,  Y_{r}(L )^{+} \otimes {\mathbb C}) ) \\
 =
 {\rm Hom}_{{\rm Gal}(L/F)}(  \oplus_{i} \  (V \otimes  W_{i} )^{\vee} ,  Y_{r}(L )^{+} \otimes {\mathbb C}) )
 \end{array}   \]
 so that
 \[ R(  \Res_{{\rm Gal}(L/K)}^{{\rm Gal}(L/F)}(V),f_{L}^{r})  =  \prod_{i}  \  R(  V  \otimes W_{i}  ,   f_{L}^{r})  \]
 and
\[   {\cal R}_{f_{L }^{r}}(  \Res_{{\rm Gal}(L/K)}^{{\rm Gal}(L/F)}(V)) = \prod_{i}  \    {\cal R}_{f_{L }^{r}}(  V  \otimes  W_{i}  ) .  \]

Let $H \subseteq G$ be finite groups with $G$ abelian. It will suffice to consider the case in which $G/H$ is cyclic of order $n$ generated by $gH$.
Let $W \otimes {\mathbb Q}$ be a free ${\mathbb Q}[G]$-module with basis $v_{1}, \ldots , v_{r}$. Then $W \otimes {\mathbb Q}$ is a free ${\mathbb Q}[H]$-module with basis $\{  g^{a}v_{i}  \  |  \  0 \leq a \leq n-1, \  1 \leq i \leq r \}$. Set ${\cal S} =  \{ 0, \ldots , n-1 \} \times \{ 1 , \ldots , r \}$; then for $\underline{u} = (a , i) \in {\cal S}$, we set $e_{\underline{u}} =  g^{a}v_{i}$. If $\tilde{\alpha}  \in     {\rm End}_{{\mathbb Q}[H]}( W   \otimes {\mathbb Q} )$ we may write
\[    \tilde{\alpha}(e_{\underline{w}}  )  =    \sum_{\underline{u}}  \  A_{ \underline{u}. \underline{w}}  e_{\underline{u}}    \]
so that $A$ is an $nr \times nr$ matrix with entries in ${\mathbb Q}[H]$.

Now consider the induced ${\mathbb Q}[G]$-module $\Ind_{H}^{G}( W \otimes   {\mathbb Q})$, which is a free  ${\mathbb Q}[G]$-module on the basis $\{  1  \otimes_{H}  e_{\underline{u}}  \  |  \ \underline{u} \in {\cal S} \}$. Hence the $nr \times nr$ matrix, with entries in ${\mathbb Q}[G]$, for $1 \otimes_{H}  \tilde{\alpha}$ with respect to this basis is the image of $A$ under the canonical inclusion of $\phi_{H,G} :  {\mathbb Q}[H]   \ra  {\mathbb Q}[G]$. In particular
\[   \phi_{H,G}(  \Det_{ {\mathbb Q}[H] }( \tilde{\alpha})) =  \Det_{  {\mathbb Q}[G]}( {\mathbb Q}[G] \otimes_{ {\mathbb Q}[H]   }  \tilde{\alpha})   \]
and, by induction on $[G:H]$, this relation is true for an arbitrary inclusion $H \subseteq G$ of finite abelian groups.

This discussion yields the following result:

\begin{prop} \label{2.6}
Suppose that $F \subseteq K \subseteq L$ is a tower of number fields with $L/F$ abelian. Then, in the notation of \S\ref{2.1},  the canonical inclusion
\[    \phi_{K/F} :  {\mathbb Q}[{\rm Gal}(L/K)]  \ra    {\mathbb Q}[{\rm Gal}(L/F)]     \]
maps $ {\cal J}_{L/K}^{r}$ onto the $    {\mathbb Z}[1/2][{\rm Gal}(L/K)]$-submodule
\[      {\mathbb Z}[1/2][{\rm Gal}(L/K)]  \langle \Det_{  {\mathbb Q}[{\rm Gal}(L/F)]}( {\mathbb Q}[{\rm Gal}(L/F)] \otimes_{ {\mathbb Q}[{\rm Gal}(L/K)]   }  (\alpha  \oplus 1) ) \inv(\hat{{\cal R}}_{f_{L}^{r} })^{-1}    \rangle  .  \]
Here, in terms of Proposition \ref{1.5},  $ \hat{{\cal R}}_{f_{L}^{r} } \in {\mathbb Q}[ {\rm Gal}(L/F)]\st$ is given by
\[   \hat{{\cal R}}_{f_{L}^{r} }(V)  =  {\cal R}_{f_{L}^{r} }( V  \otimes  \Ind_{{\rm Gal}(L/K)}^{{\rm Gal}(L/F)}(1)) \]
and $  \alpha \in  {\rm End}_{{\mathbb Q}[{\rm Gal}(L/K)]}( Y_{r}(L)^{+} \otimes {\mathbb Q} )$ runs through endomorphisms satisfying the integrality condition of \S\ref{2.1}.
\end{prop}

\subsection{Behaviour under fixed-point maps} \label{2.7}

As in \S\ref{2.4}, suppose that $F \subseteq K \subseteq L$ is a tower of number fields with $L/F$ abelian. Let $e_{L/K} =    [L:K]^{-1}( \sum_{y \in {\rm Gal}(L/K)}  \  y ) $ denote the idempotent associated with the subgroup ${\rm Gal}(L/K) $.
 There is a homomorphism of unital rings of the form
 \[   \lambda_{K/F} : {\mathbb Q}[{\rm Gal}(K/F)]     \ra   {\mathbb Q}[{\rm Gal}(L/F)]   \]
 given,  for $z \in {\rm Gal}(L/F)$, by the formula
\[   \lambda_{K/F}(  z {\rm Gal}(L/K)) =  (1 - e_{L/K}) + z  \cdot  e_{L/K}    \in  {\mathbb Q}[{\rm Gal}(L/F)] . \]
From Proposition \ref{1.5} it is easy to see that in terms of group characters
\[    {\rm Map}( \chgrp{{\rm Gal}}(K/F) ,  \overline{{\mathbb Q}})   \ra   {\rm Map}( \chgrp{{\rm Gal}}(L/F) ,  \overline{{\mathbb Q}})    \]
this sends a function $h$ on $ \chgrp{{\rm Gal}}(K/F)$ to the function $h'$ given by
\[    h'(\chi)  =  \left\{
\begin{array}{ll}
h(\chi_{1})  &  {\rm if}  \   \Inf_{{\rm Gal}(K/F)}^{{\rm Gal}(L/F)}(\chi_{1}) =  \chi  ,  \\
1  & {\rm otherwise}.
\end{array}
\right. \]

Sending a complex representation $V$ of ${\rm Gal}(L/F)$ to its ${\rm Gal}(L/K)$-fixed points $V^{{\rm Gal}(L/K)}$ gives a homomorphism
\[ \mathrm{Fix} : R({\rm Gal}(L/F)) \ra  R({\rm Gal}(K/F)) .   \]
In terms of one-dimensional respresentations (i.e. characters) the above condition $  \Inf_{{\rm Gal}(K/F)}^{{\rm Gal}(L/F)}(\chi_{1}) =  \chi $ is equivalent to $\mathrm{Fix}(\chi) = \chi_{1}$.

Let $V$ be a one-dimensional complex representation of ${\rm Gal}(L/F)$ fixed by ${\rm Gal}(L/K)$. Then we have isomorphisms of the form
\[  \begin{array}{l}
{\rm Hom}_{{\rm Gal}(L/F)}(( V^{{\rm Gal}(L/K)})^{\vee} ,  Y_{r}(L)^{+}  \otimes {\mathbb C})  \\
=  {\rm Hom}_{{\rm Gal}(K/F)}( V^{\vee} ,  (  Y_{r}(L)^{{\rm Gal}(L/K)})^{+}  \otimes {\mathbb C})  \\
=  {\rm Hom}_{{\rm Gal}(K/F)}(V^{\vee} ,  Y_{r}(K)^{+}  \otimes {\mathbb C})
\end{array}  \]
and, by invariance of $L$-functions under inflation, $L_{F}^{*}(r, V) =  L_{F}^{*}(r,  V^{{\rm Gal}(L/K)})$.
Therefore, by the discussion of \S\ref{2.3},
\[   {\cal R}_{f_{L }^{r}}(V)  =  {\cal R}_{f_{K }^{r}}( V^{{\rm Gal}(L/K)} )   .  \]
On the other hand, if $V^{{\rm Gal}(L/K)} = 0$ then $ {\cal R}_{f_{K }^{r}}( V^{{\rm Gal}(L/K)} ) = 1$ since both $L_{F}^{*}(r, 0 )$ and the determinant of the identity map of the trivial vector space are equal to one.
This establishes the formula
\[   \lambda_{K/F}(  {\cal R}_{f_{K }^{r}}) =   (1 - e_{L/K}) +     {\cal R}_{f_{L }^{r}} \cdot  e_{L/K}  . \]

Now consider an endomorphism
\[   \alpha \in {\rm End}_{  {\mathbb Q}[{\rm Gal}(K/F)]}( Y_{r}(K)^{+}  \otimes {\mathbb Q} )  \]
satisfying the integrality condition of \S\ref{2.1}
\[  \alpha f_{r , K}( K_{1-2r}( {\cal O}_{K}[t^{\pm 1}]) ) \subseteq  Y_{r}(K)^{+}  \cong  (Y_{r}(L)^{+})^{{\rm Gal}(L/K)}  .  \]
Let $v_{1}, v_{2}, \ldots , v_{d}$ be a ${\mathbb Z}[1/2][ {\rm Gal}(L/F) ]$-basis of $Y_{r}(L)[1/2]$
so that a \linebreak ${\mathbb Z}[1/2][ {\rm Gal}(K/F) ]$-basis of the subspace $(Y_{r}(L)^{+})^{{\rm Gal}(L/K)}[1/2] \cong Y_{r}(K)[1/2]$ is given by $  \{    (\sum_{y \in {\rm Gal}(L/K)}  \  y )v_{i}  \  |  \   1 \leq i \leq d \}     $. To construct the generators of ${\cal J}_{K/F}^{r} $, as in \S\ref{2.1}, we must calculate the determinant of $\alpha \oplus 1$ on $Y_{r}(K)^{+}  \otimes {\mathbb Q}  \oplus Y_{r}(K)^{-}  \otimes {\mathbb Q} =  Y_{r}(K)  \otimes {\mathbb Q}$ with respect to the basis  $  \{    (\sum_{y \in {\rm Gal}(L/K)}  \  y )v_{i}  \}$
and divide by  $ \inv(   {\cal R}_{f_{K}^{r}})$.

Let $\hat{\alpha} \in {\rm End}_{  {\mathbb Q}[{\rm Gal}(L/F)]}( Y_{r}(L)  \otimes {\mathbb Q} ) $ be given by $\alpha$ on $Y_{r}(L)^{{\rm Gal}(L/F)} \otimes {\mathbb Q}$ and the identity on $(1 - e_{L/K})Y_{r}(L)  \otimes {\mathbb Q}$. Hence $\hat{\alpha}$ satisfies the integrality condition
\[  \hat{\alpha} \cdot   f_{L}^{r}( K_{1-2r}({\cal O}_{L}[t^{\pm 1}]))^{{\rm Gal}(L/F)}      \subseteq Y_{r}(L)^{{\rm Gal}(L/F)} ,  \]
because, as in \S\ref{2.3}, $f_{K}^{r}$ may be assumed to extend to $f_{L}^{r}$.
Therefore
\[   e_{L/K}   \frac{{\rm det}(\hat{\alpha})}{ \inv(   {\cal R}_{f_{L}^{r}})} \in  e_{L/K}  {\cal J}_{L/F}^{r}
\subset {\mathbb Q}[ {\rm Gal}(L/F) ] .  \]
On the other hand it is clear that $\lambda_{K/F}( {\rm det}(\alpha \oplus 1)) =  {\rm det}(\hat{\alpha})$.

This discussion has established the following result.

\begin{prop} \label{2.8}
Suppose that $F \subseteq K \subseteq L$ is a tower of number fields with $L/F$ abelian and let
 \[   \lambda_{K/F} : {\mathbb Q}[{\rm Gal}(K/F)]     \ra   {\mathbb Q}[{\rm Gal}(L/F)]   \]
 denote the unital ring homomorphism of \S\ref{2.7}. Then
 \[ \lambda_{K/F}(  {\cal J}_{K/F}^{r}  )  \subseteq  (1- e_{L/K})  {\mathbb Q}[{\rm Gal}(L/F)] +  e_{L/K} {\cal J}_{L/F}^{r}   .   \]
\end{prop}

\subsection{Behaviour under corestriction maps} \label{2.9}

As in \S\ref{2.4}, suppose that $F \subseteq K \subseteq L$ is a tower of number fields with $L/F$ abelian. There is an additive homomorphism of the form
 \[   \iota_{K/F} : {\mathbb Q}[{\rm Gal}(L/F)]     \ra   {\mathbb Q}[{\rm Gal}(L/K)]   \]
called the transfer or corestriction map. In terms of
Proposition \ref{1.5} it is induced by the induction of representations
\[  {\rm Ind}_{{\rm Gal}(L/K)}^{{\rm Gal}(L/F)} :  R({\rm Gal}(L/K)) \ra    R({\rm Gal}(L/F))  .  \]
That is, the image $  \iota_{K/F} (h)$ of $h \in   {\rm Hom}_{\Omega_{\mathbb Q}}( R({\rm Gal}(L/F)) ,  \overline{{\mathbb Q}}) $ is given by
\[   \iota_{K/F} (h)(V) = h( {\rm Ind}_{{\rm Gal}(L/K)}^{{\rm Gal}(L/F)}(V) )  .  \]

By Frobenius reciprocity, for each $V \in R({\rm Gal})(L/K) )$ there is an isomorphism
\[  \begin{array}{l}
{\rm Hom}_{{\rm Gal}(L/F)}(( {\rm Ind}_{{\rm Gal}(L/K)}^{{\rm Gal}(L/F)})^{\vee}  ,  Y_{r}(L)^{+}  \otimes {\mathbb C})  \\
=  {\rm Hom}_{{\rm Gal}(L/K)}( V^{\vee} ,   Y_{r}(L)^{+}  \otimes {\mathbb C})  .
\end{array}  \]
Also $L_{F}^{*}(  r ,   {\rm Ind}_{{\rm Gal}(L/K)}^{{\rm Gal}(L/F)}(V)) =  L_{K}^{*}(r , V)$ so that
\[    \iota_{K/F}(  {\cal R}_{f_{L}^{r}} )  =      {\cal R}_{f_{L}^{r}}   .   \]

Now consider an endomorphism
\[   \alpha \in {\rm End}_{  {\mathbb Q}[{\rm Gal}(L/F)]}( Y_{r}(L)^{+}  \otimes {\mathbb Q} )  \]
satisfying the integrality condition of \S\ref{2.1}
\[  \alpha f_{r , L}( K_{1-2r}( {\cal O}_{L}[t^{\pm 1}]) ) \subseteq  Y_{r}(L)^{+} .  \]
Then it is straightforward to see from Proposition \ref{1.5} that the determinant of $\alpha \oplus 1$ as a map of ${\mathbb Q}[{\rm Gal}(L/F)]$-modules ${\rm det}_{{\mathbb Q}[{\rm Gal}(L/F)]}(\alpha \oplus 1)$
is mapped to  ${\rm det}_{{\mathbb Q}[{\rm Gal}(L/K)]}(\alpha \oplus 1)$, the determinant of $\alpha \oplus 1$ as  $\Q[\gal{L/K}]$-modules.

This discussion has established the following result. (This result has a problem. See Section~\ref{error}.)

\begin{prop} \label{2.10}

Suppose that $F \subseteq K \subseteq L$ is a tower of number fields with $L/F$ abelian, and let
 \[   \iota_{K/F} : {\mathbb Q}[{\rm Gal}(L/F)]     \ra   {\mathbb Q}[{\rm Gal}(L/K)]   \]
denote the additive homomorphism of \S\ref{2.9}. Then
 \[ \iota_{K/F}(  {\cal J}_{L/F}^{r}  )  \subseteq        {\cal J}_{L/K}^{r}   .   \]
\end{prop}

\subsection{} \label{nat irrats}

We can now explain the second example in \S \ref{nat examples}, \ie Proposition \ref{nat irrats example}. Let us work more generally to begin with. $E$ and $F$ can be any number fields, and we suppose we have a diagram
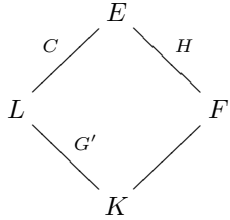
\[ \xymatrix{
& E \ar@{-}[dl]_C \ar@{-}[dr]^H & \\
L \ar@{-}[dr]^{G'} & & F \ar@{-}[dl] \\
& K & 
} \]
satisfying the following: $E/K$ is Galois (though not necessarily abelian), $LF = E$, $L \cap F = K$, the extension $L/K$ is abelian (and hence so is $E/F$), and $L/K$ and $E/F$ satisfy the Stark conjecture. We let $G = \gal{E/K}$, and the Galois groups of the other Galois extensions are marked in the diagram. We observe that $C$ need not be abelian here.

Owing to the natural isomorphism $G/C \ra H$, each character $\psi \in \chgrp{H}$ extends to a unique one-dimensional representation $\widehat{\psi} : G \ra \C\st$ which is trivial on $C$. Denote by $\irrch{G}$ the set of irreducible characters of $G$. Then having chosen a $\Q[G]$-module isomorphism $f$ as in \S \ref{stark conjectures}, we can define an element $\BC{f} \in \C[H]\st$ by
\[ \BC{f} = \prod_{\chi \in \irrch{G} \setm \{1\}} \left(\sum_{\psi \in \chgrp{H}} \VR{f}_{E/K}(\chi \widehat{\psi})^{d_\chi} e_\psi \right) ,\]
where for a character $\chi$ of $G$, $d_\chi$ is the multiplicity of the trivial character of $H$ in $\Res_H^G(\chi)$. We have opted to denote by $\VR{f}_{E/K}$ the group-ring element $\mathcal{R}_{f_E}$ defined in \S \ref{2.1}, to emphasize which extension is being considered.

The following lemma shows that the group-ring element $\VR{f}_{E/F}$ for the extension $E/F$ is related, via $\BC{f}$, to the corresponding element for the extension $L/K$.

\begin{lemma} \label{base change}
$\BC{f}$ has rational coefficients, and the image of $\VR{f}_{E/F}$ under the isomorphism $\Phi : \Q[H] \ra \Q[G']$ is
\[ \VR{f'}_{L/K} \Phi(\BC{f}) ,\]
where $f'$ is the $\Q[G']$-module isomorphism making diagram (\ref{f chosen to commute}) commute.
\end{lemma}

The proof of the lemma is little more than a combination of \S \ref{2.3} and \S \ref{2.9}.

In the situation of Proposition \ref{nat irrats example} (with $L = E^+$ and $K = \Q$ now) we find that the element $2\halfst$ occurring there is just $\inv(\BC{f})^{-1}$ (for any choice of $f$ in this case). Indeed, let $\rho \in \chgrp{G}$ be the unique non-trivial character extending the trivial character of $H$. Then the only $\chi \in \irrch{G} \setm \{1\}$ with $d_\chi \not= 0$ is $\rho$, and $d_\rho = 1$, so
\begin{eqnarray*}
\BC{f} &=& \sum_{\psi \in \chgrp{H}} \VR{f}_{E/\Q}(\rho\widehat{\psi}) e_\psi \\
&=& \sum_{\substack{\psi \in \chgrp{G} \\ \psi \mathrm{ even}}} \VR{f}_{E/\Q}(\rho \psi) e_{\psi|_H} .
\end{eqnarray*}
However, for $\psi$ even, $\rho \psi$ is odd so that $\VR{f}(\rho \psi) = L_{E/\Q,S}(0,\rho \psi)^{-1}$. Using the easily verified fact that $(1-c)\halfst = \stick_{E/\Q,S}$, where $c \in G$ is complex conjugation, we see that $L_{E/\Q,S}(0,\rho \psi) = 2\psi|_H(\inv \halfst)$, from which the assertion follows.

Applying Lemma \ref{base change} now justifies the appearance of $2\Phi_n(\halfst_n)$ in Proposition \ref{nat irrats example}.

\section{The passage to non-abelian groups}

\subsection{} \label{3.1} In this section we shall use the Explicit Brauer Induction constructions of \cite[pp.138--147]{snaith:ebi} to pass from finite abelian Galois groups to the non-abelian case.

Let $G$ be a finite group and consider the additive homomorphism
\[    \sum_{H \subseteq G}  \  \Ind_{H}^{G} \Inf_{H\ab}^{H}  :  \oplus_{H \subseteq G}  \  R(H\ab)  \ra  R(G) .\]

Let $N \lhd G$ be a normal subgroup and let $\pi : G \ra  G/N$ denote the quotient homomorphism.

Define a homomorphism
\[ \alpha_{G,N} :   \oplus_{J \subseteq G/N}  \  R(J\ab)   \ra  \oplus_{H \subseteq G}  \  R(H\ab)   \]
to be the homomorphism which sends the $J$-component $R(J\ab) $ to the $H = \pi^{-1}(J)$-component $R( \pi^{-1}(J)\ab) $ via the map
\[   \Inf_{J\ab}^{\pi^{-1}(J)\ab}(R(J\ab))     \ra   R(\pi^{-1}(J)\ab) .  \]

\begin{lemma} \label{3.2}
In the notation of \S\ref{3.1} the following diagram commutes:
\[ \xymatrix{
\bigoplus_{J \con G/N} R(J\ab) \ar[r] \ar[d]^{\alpha_{G,N}} & R(G/N) \ar[d]^{\Inf_{G/N}^G} \\
\bigoplus_{H \con G} R(H\ab) \ar[r] & R(G) .
} \]
\end{lemma}

\begin{proof}
Since the kernel of $\pi^{-1}(J) \ra J$ and that of $\pi : G \ra  G/N$ coincide, both being equal to $N$, we have
\[    \Inf_{G/N}^{G} \Ind_{J}^{G/N}  =  \Ind_{\pi^{-1}(J) }^{G} \Inf_{J}^{ \pi^{-1}(J)   } .      \]
Therefore, given a character $\phi : J\ab \ra \overline{{\mathbb Q}}\st$ in the $J$-coordinate, we have
\[  \begin{array}{ll}
 \Ind_{ \pi^{-1}(J) }^{G} \Inf_{\pi^{-1}(J)\ab}^{\pi^{-1}(J)}  \alpha_{G,N}(\phi) & =     \Ind_{ \pi^{-1}(J) }^{G}\Inf_{\pi^{-1}(J)\ab}^{\pi^{-1}(J)}  \Inf_{J\ab}^{\pi^{-1}(J)\ab}( \phi)  \\
 & =     \Ind_{ \pi^{-1}(J) }^{G}\Inf_{J}^{\pi^{-1}(J)}  \Inf_{J\ab}^{J}( \phi) \\
 & =     \Inf_{G/N}^{G} \Ind_{J}^{G/N}  \Inf_{J\ab}^{J}( \phi) ,
 \end{array} \]
 as required.
\end{proof}

\subsection{} \label{3.3} The homomorphism of \S\ref{3.1} is invariant under group conjugation and therefore induces an additive homomorphism of the form
 \[   B_{G}  : ( \oplus_{H \subseteq G}  \  R(H\ab))_{G}  \ra  R(G) \]
 where $X_{G}$ denotes the coinvariants of the conjugation $G$-action. This homomorphism is a split surjection whose right inverse is given by the Explicit Brauer Induction homomorphism
 \[ A_{G} : R(G) \ra   ( \oplus_{H \subseteq G}  \  R(H\ab))_{G} \]
 constructed in \cite[Section 4.5.16]{snaith:ebi}. We shall be interested in the dual homomorphisms (\cite[Section 4.5.20]{snaith:ebi})
\[   B_{G}^{*} : {\rm Hom}_{\Omega_{{\mathbb Q}}}( R( G) ,
\overline{{\mathbb Q}})   \ra    ( \oplus_{H \subseteq G} \ {\rm Hom}_{\Omega_{{\mathbb Q}}}( R( H\ab)    , \overline{{\mathbb Q}}) )^{G} \]
 and
 \[ A_{G}^{*} :   ( \oplus_{H \subseteq G} \ {\rm Hom}_{\Omega_{{\mathbb Q}}}( R( H\ab)    , \overline{{\mathbb Q}}) )^{G}    \ra {\rm Hom}_{\Omega_{{\mathbb Q}}}( R( G) ,
\overline{{\mathbb Q}})   \]
 where $X^{G}$ denotes the subgroup of $G$-invariants.

 As in \cite[Def.4.5.4]{snaith:ebi}, denote by ${\mathbb Q}\{G\}$ the rational vector space whose basis consists of the conjugacy classes of $G$. There is an isomorphism (\cite[Prop.4.5.14]{snaith:ebi})
 \[ \psi :  {\mathbb Q}\{G\}   \stackrel{\cong}{\ra} {\rm Hom}_{\Omega_{{\mathbb Q}}}( R( G) ,
\overline{{\mathbb Q}})  \]
 given by the formula $\psi( \sum_{\gamma} \ m_{\gamma}  \gamma)( \rho) = \sum_{\gamma} \ m_{\gamma} {\rm Trace}(\rho(\gamma))$.

 When $G$ is abelian, we have $ {\mathbb Q}\{G\} =  {\mathbb Q}[G]$ and under the identification
 \[   {\rm Hom}_{\Omega_{{\mathbb Q}}}( R( G) ,
\overline{{\mathbb Q}})   =    {\rm Map}_{\Omega_{{\mathbb Q}}}(\chgrp{G} ,  \overline{{\mathbb Q}}) \]
 of Proposition \ref{1.5} we have $\psi(g) = ( \chi \mapsto \chi(g))$, which is a ring isomorphism inverse to $\lambda_{G}$.

\section{${\cal J}_{\ef/\gf}^{r}$ in general } \label{4.1}

 Let $G$ denote the Galois group of a finite Galois extension $E/F$ of number fields. Hence each subgroup of $G$ has the form $H =   {\rm Gal}(E/E^{H})$, whose abelianization  is $H\ab =  {\rm Gal}(E^{[H,H]}/E^{H})$ where $[H,H]$ is the commutator subgroup of $H$. For each integer $r= 0, -1, -2, -3, \ldots$, we have the canonical fractional Galois ideal ${\cal J}_{E^{[H,H]}/E^{H}}^{r} \subseteq{\mathbb Q}[H\ab]$ as defined in \S\ref{2.1}.

 \begin{definition} \label{4.2}
 In the notation of \S\ref{4.1}, define a subgroup ${\cal J}_{\ef/\gf}^{r}$ of ${\mathbb Q}\{G\}$ by
 \[   {\cal J}_{\ef/\gf}^{r}  =  (B_{G}^{*})^{-1}(  \oplus_{H \subseteq G} \   {\cal J}_{E^{[H,H]}/E^{H}}^{r})  .  \]
 \end{definition}

 \begin{lemma} \label{4.3}
 In \S\ref{4.1} and Definition \ref{4.2}, when $G = {\rm Gal}(\ef/\gf)$ is abelian then ${\cal J}_{\ef/\gf}^{r}$ coincides with the canonical fractional Galois ideal of \S\ref{2.1}.
 \end{lemma}

\begin{proof}
 The $H$-component of $B_{G}^{*}$ has the form
 \[   {\mathbb Q}[ {\rm Gal}(\ef/\gf)]  \stackrel{i_{E^{H}/F}}{\ra } {\mathbb Q}[ {\rm Gal}(E/E^{H})]     \stackrel{\pi_{E/E^{[H,H]}}}{\ra}    {\mathbb Q}[ {\rm Gal}(E^{[H,H]}/E^{H})]    \]
   which maps ${\cal J}_{\ef/\gf}^{r} $ to ${\cal J}_{E^{[H,H]}/E^{H}}^{r} $ by Proposition \ref{2.4} and Proposition \ref{2.10} so that
   \[   {\cal J}_{\ef/\gf}^{r}  \subseteq    (B_{G}^{*})^{-1}(   \oplus_{H \subseteq G} \   {\cal J}_{E^{[H,H]}/E^{H}}^{r} )   .  \]
   On the other hand, the $G$-component of $B_{G}^{*}$ is the identity map from ${\mathbb Q}[G]$ to itself. Therefore if $z \in {\mathbb Q}[G] \setm {\cal J}_{\ef/\gf}^{r}$ then $B_{G}^{*}(z) \not\in    \oplus_{H \subseteq G} \   {\cal J}_{E^{[H,H]}/E^{H}}^{r}$, as required.
\end{proof}

\begin{prop} \label{4.4}
Suppose that $F \subseteq K \subseteq L$ is a tower of finite extensions of number fields with $L/F$ and $K/F$ Galois. Then, for $r=0, -1, -2, -3, \ldots $, the canonical homomorphism
\[    \pi_{L/K} : {\mathbb Q}\{ {\rm Gal}(L/F)  \}  \ra    {\mathbb Q}\{ {\rm Gal}(K/F)  \}   \]
satisfies $ \pi_{L/K}( {\cal J}_{\ef/\gf}^{r}  ) \subseteq  {\cal J}_{K/F}^{r} $.
\end{prop}

\begin{proof}
This follows immediately from Proposition \ref{2.4}, Lemma \ref{3.2} and Definition \ref{4.2}.
\end{proof}

\begin{definition} \label{4.5}
Let $F$ be a number field and $L/F$ a (possibly infinite) Galois extension with Galois group $G = {\rm Gal}(L/F)$. For $r=0, -1, -2, -3, \ldots$ define ${\cal J}_{\ef/\gf}^{r} $ to be the abelian group
\[ {\cal J}_{\ef/\gf}^{r}   =  \proo{H} {\cal J}_{L^{H}/F}^{r} ,\]
where $H$ runs through the open normal subgroups of $G$.
\end{definition}

\section{${\cal J}_{\ef/\gf}^{r} $ and the annihilation of \\ $ H_{{\rm \acute{e}t}}^{2}({\rm Spec}({\cal O}_{L,S}) , {\mathbb Z}_{\prm}(1-r))$}

%\footnote{When $r=0$ the relation of this cohomology group to the ideal class is well known. Recent results of Voevodsky, Suslin, Rost et al show that for $r = -1,-2,-3, \ldots$ it is naturally isomorphic to the Quillen algebraic K-group $K_{-2r}({\cal O}_{L,S}) \otimes {\mathbb Z}_{\prm}$.}

\subsection{} \label{5.1} Let $\prm$ be an odd prime. We continue to assume the Stark conjecture as stated in \S\ref{stark conjectures} for $r = 0, -1, -2, -3, \ldots $. Replacing ${\mathbb Q}$ by ${\mathbb Q}_{\prm}$ in \S\ref{2.1} and Definition \ref{4.2} we may associate a finitely generated ${\mathbb Z}_{\prm}$-submodule of ${\mathbb Q}_{\prm}\{ {\rm Gal}(\ef/\gf) \}$, again denoted by ${\cal J}_{\ef/\gf}^{r} $, to any finite extension $\ef/\gf$ of number fields.

In this section we are going to explain a conjectural procedure to pass from ${\cal J}_{\ef/\gf}^{r} $ to the construction of elements in the annihilator ideal of the \'{e}tale cohomology of the ring of $S$-integers of $E$,
\[     {\rm ann}_{{\mathbb Z}_{\prm}[G(\ef/\gf)]}( H_{{\rm \acute{e}t}}^{2}(
{\rm Spec}({\cal O}_{E,S(E)}) , {\mathbb Z}_{\prm}(1-r))) ,\]
where $S$ denotes a finite set of primes of $F$ including all archimedean primes and all finite primes which ramify in $\ef/\gf$, and $S(E)$ denotes all the primes of $E$ over those in $S$. This conjectural procedure was first described in \cite[Thm.8.1]{snaith:stark}.

We shall restrict ourselves to the case when $r = -1,-2, -3, \ldots$.  In several ways this is a simplification over the case when $r=0$. In this case $H_{{\rm \acute{e}t}}^{1}(
{\rm Spec}({\cal O}_{E, S(E)}) , {\mathbb Z}_{\prm}(1-r))$ is independent of $S(E)$, while it is related to the group of $S(E)$-units when $r=0$. Also,  when $ r \leq -1$,  $H_{{\rm \acute{e}t}}^{2}(
{\rm Spec}({\cal O}_{E,S(E)}) , {\mathbb Z}_{\prm}(1-r))$ is a subgroup of the corresponding cohomology group when $S(E)$ is enlarged to $S'(E)$, but when $r=0$ the class-group of ${\cal O}_{E, S'(E)}$ is a quotient of that of ${\cal O}_{E , S(E)}$. Furthermore (see \cite{buckingham:frac}, \cite{tate:stark}), there are subtleties concerning whether or not to use the $S$-modified $L$-function in \S \ref{notation} when $r=0$, while for $r \leq -1$ this is immaterial.

When $r=0$ the annihilator procedure is similar to the other cases but the additional complications have prompted us to omit this case.

Write $G = {\rm Gal}(\ef/\gf) $, and for each subgroup $H = {\rm Gal}(E/E^{H})   \subseteq G$ \linebreak let $S(E^{H})$ denote the set of primes of $E^{H}$ above those of $S$. Then $H\ab = {\rm Gal}(E^{[H,H]}/E^{H})$ where $[H,H]$ denotes the commutator subgroup of $H$. The following conjecture originated in \cite{snaith:equiv,snaith:rel,snaith:stark}.

\begin{conj} \label{5.2}
In the notation of \S\ref{5.1}, when $r = -1, -2, -3, \ldots$,

(i)  \   {\bf Integrality:}
\[  {\cal J}_{E^{[H,H]}/E^{H}}^{r} \cdot    {\rm ann}_{{\mathbb Z}_{\prm}[H\ab]}( {\rm Tors} H_{{\rm \acute{e}t}}^{1}(
{\rm Spec}({\cal O}_{E^{[H,H]},S}) , {\mathbb Z}_{\prm}(1-r)))   \subseteq {\mathbb Z}_{\prm}[H\ab]  .  \]

(ii)  \  {\bf Annihilation:}
\[   \begin{array}{l} {\cal J}_{E^{[H,H]}/E^{H}}^{r} \cdot    {\rm ann}_{{\mathbb Z}_{\prm}[H\ab]}( {\rm Tors} H_{{\rm \acute{e}t}}^{1}(
{\rm Spec}({\cal O}_{E^{[H,H]},S}) , {\mathbb Z}_{\prm}(1-r)))  \\
 \subseteq {\rm ann}_{{\mathbb Z}_{\prm}[H\ab]}( H_{{\rm \acute{e}t}}^{2}(
{\rm Spec}({\cal O}_{E^{[H,H]},S}) , {\mathbb Z}_{\prm}(1-r))) .
 \end{array}   \]
\end{conj}

(We have adopted the shorthand: $\oh_{E^{[H,H]},S} = \oh_{E^{[H,H]},S(E^{[H,H]})}$.)

\subsection{Evidence} \label{5.3}

Part (i) of Conjecture \ref{5.2} is analogous to the Stickelberger integrality, which is described in \cite[Section 2.2]{snaith:stark}. Stickelberger integrality was proven in certain totally real cases in \cite{kl:padicl,cs:padic,cassou-nogues:valeurs,dr:abelianlfunctions}, for $\lint = 0$. In general, when $r=0$, it is part of the Brumer conjecture \cite{brumer:units}. The novelty of part (ii) of Conjecture \ref{5.2}, when it was introduced in \cite{snaith:rel} and \cite{snaith:stark}, was the annihilator prediction when the $L$-function vanishes at $s=r$. For the part of the fractional ideal corresponding to characters whose $L$-functions are non-zero at $s = \lint$, generated by the higher Stickelberger element at $s = \lint$, part (ii) is the conjecture of \cite{cs:stickel}.

Let us consider the cyclotomic example $\hJ{\lint}{L/\Q}$ ($\lint <0$) when $L = \Q(\zeta)$ for some root of unity $\zeta$, and suppose $\prm$ is an odd prime dividing the order of $\zeta$. In this case, $\hJ{\lint}{L/\Q}$ splits into plus and minus parts for complex conjugation, \ie
\[ \hJ{\lint}{L/\Q} = e_+^\lint \hJ{\lint}{L/\Q} \oplus e_-^\lint \hJ{\lint}{L/\Q} ,\]
where $e_+^\lint = \frac{1}{2}(1 + (-1)^\lint c)$, $e_-^\lint = \frac{1}{2}(1 - (-1)^\lint c)$ and $c \in G = \gal{L/\Q}$ is complex conjugation. By the proof of \cite[Theorem 6.1]{snaith:stark}, $e_-^\lint \hJ{\lint}{L/\Q}$ is generated by the Stickelberger element $\stick_{L/\Q,S}(\lint)$ defined in terms of $L$-function values at $s = \lint$. However, by \cite{dr:abelianlfunctions},
\[ \ann_{\Zprm[G]}(\tors (\het^1(\spec \oh_{L,S},\Zprm(1-\lint)))) \stick_{L/\Q,S}(\lint) \con \Zprm[G] .\]
Further, the proof of \cite[Theorem 7.6]{snaith:stark} shows that $e_+^\lint \hJ{\lint}{L/\Q} \con \Zprm[G]$. In fact, \cite[Theorem 6.1]{snaith:stark} also shows that part (ii) of Conjecture \ref{5.2} holds in this case (with $E = \Q$ and $H = G$), the intersection ``$\cap \Zprm[G]$'' found in the statement of that theorem being unnecessary.

Turning now to the case $\lint = 0$, with the field $\ef_n$ as in \S \ref{nat examples}, we have a similar scenario for $\iJ{\ef_n/\Q,S}$, where $S = \{\infty,\prm\}$. Indeed, we see from (\ref{recap full j desc}) that $\iJ{\ef_n/\Q,S}$ again splits into plus and minus parts, with the minus part being generated by the Stickelberger element $\stick_{\ef_n/\Q,S}$ defined at $s = 0$. Stickelberger's theorem then implies that
\[ \ann_{\Zprm[G_n]}(\rou{\ef_n}) e_- \iJ{\ef_n/\Q,S} \con \Zprm[G_n] ,\]
and $e_+ \iJ{\ef_n/\Q,S}$ is already in $\Zprm[G_n]$. The roles of the plus and minus parts of $\iJ{\ef_n/\Q,S}$ will become clear in \S \ref{commutative example} below.

\subsubsection{An Iwasawa-theoretic example} \label{commutative example}

(\ref{recap full j desc}) can be used to provide an example of the relationship of $\iJ{\ef_n/\Q,S}$ to Iwasawa theory, with an inverse limit of the $\iJ{\ef_n/\Q,S}$ over $n$ giving rise, in a suitable way, to Fitting ideals of both the plus and minus parts of an inverse limit of class-groups (Proposition \ref{iJ and fitt of class}). Given $n \geq 0$, let $\cycq{n}/\Q$ be the degree $\prm^n$ subextension of the (unique) $\Zprm$-extension $\cycq{\infty}$ of $\Q$. We then have the field diagram
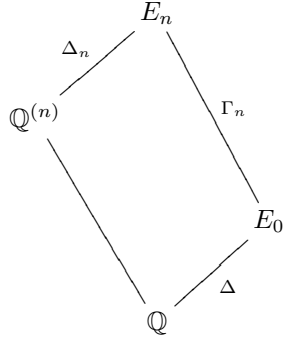
\[ \xymatrix{
& \ef_n \ar@{-}[dl]_{\Delta_n} \ar@{-}[ddr]^{\Gamma_n} & \\
\cycq{n} \ar@{-}[ddr] & & \\
& & \ef_0 \ar@{-}[dl]^\Delta \\
& \Q &
} \]
in which $\cycq{n} \cap \ef_0 = \Q$ and $\cycq{n} \ef_0 = \ef_n$, so that the Galois group $G_n = \gal{\ef_n/\Q}$ is the internal direct product of $\Delta_n$ and $\Gamma_n$. $S$ will denote the set of places $\{\infty,\prm\}$ of $\Q$.

By virtue of the natural isomorphism $\Delta_n \ra \Delta$, characters of $\Delta_n$ correspond to characters of $\Delta$. If $\tsc \in \chgrp{\Delta}$, we let $\tsc_n$ denote the corresponding character in $\chgrp{\Delta}_n$. Now, the idea is to view the group-ring $\C[G_n]$ as $\C[\Gamma_n][\Delta_n]$. In doing this, we can define a projection $\pi_n(\tsc) : \C[G_n] \ra \C[\Gamma_n]$ by extending $\tsc_n$ linearly (over $\C[\Gamma_n]$).

Finally, fix an isomorphism $\nu : \Cprm \ra \C$ and let $\teich : \Delta \ra \C\st$ be the composition of the Teichm\"uller character $\Delta \ra \Cprm\st$ with $\nu : \Cprm\st \ra \C\st$. Then given $\tsc \in \chgrp{\Delta}$, $\tsc^*$ will denote $\teich \tsc^{-1}$. Observe that since $\teich$ is odd, $\tsc$ is even if and only if $\tsc^*$ is odd.

\begin{prop} \label{iJ and fitt of class}
Let $\tsc \in \chgrp{\Delta}$. ($\tsc$ may be even or odd.)
\[ \fitt_{\Zprm\pwr{\Gamma_\infty}}(e_{\tsc^*} \class_\infty) = \left\{
\begin{array}{ll}
\displaystyle{\proo{n}} \Zprm \pi_n(\tsc^*)(\cycJ) & \textrm{if $\tsc \not= 1$} \\
\displaystyle{\proo{n}} \Zprm \pi_n(\tsc^*)((1 - (1+\prm)\sigma_n^{-1}) \cycJ) & \textrm{if $\tsc = 1$}
\end{array}
\right. \]
where $\sigma_n = (1+\prm,\ef_n/\Q)$.
\end{prop}

\begin{proof}
This stems from (\ref{recap full j desc}), which we reproduce for convenience:
\[ \iJ{\ef_n/\Q,S} = \frac{1}{2} e_+ \ann_{\Z[G_n]}(\oh_{\ef_n^+,S}\st/\E_n^+) \oplus \Z[G_n]\stick_{\ef_n/\Q,S} .\]

Let us deal with even characters $\tsc \in \chgrp{\Delta}$ first. For simplicity, we will assume that $\tsc \not= 1$, though in fact the case $\tsc = 1$ is similar. (\ref{recap full j desc}) tells us that for each $n \geq 0$, $\Zprm \pi_n(\tsc^*)(\cycJ) = \Zprm[\Gamma_n]\pi_n(\tsc^*)(\stick_{\ef_n/\Q,S})$. However, Iwasawa's construction of $\prm$-adic $L$-functions (see \cite{iwasawa:padicl} and \cite[Chapter 7]{wash:cyc}) shows that this lies in $\Zprm[\Gamma_n]$ and that the inverse limit of these ideals is generated by the algebraic $\prm$-adic $L$-function corresponding to the even character $\tsc$. Mazur and Wiles' proof (see \cite{mw:classfields}) of the Main Conjecture of Iwasawa theory, and later Wiles' generalization of this (see \cite{wiles:iwasawa}), show that this in turn is equal to the Fitting ideal appearing in the statement of the proposition.

Now we turn to odd characters $\tsc \in \chgrp{\Delta}$. Referring to (\ref{recap full j desc}) again, we find that
\[ \Zprm \pi_n(\tsc^*)(\cycJ) = \pi_n(\tsc^*)(\fitt_{\Zprm[G_n]}((\oh_{\ef_n^+,S}\st/\E_n^+) \teno{\Z} \Zprm)) .\]
This uses that $(\oh_{\ef_n^+,S}\st/\E_n^+) \teno{\Z} \Zprm$ is cocyclic as a $\Zprm[G_n]$-module so that, since $G_n$ is cyclic, the Fitting and annihilator ideals of $(\oh_{\ef_n^+,S}\st/\E_n^+) \teno{\Z} \Zprm$ agree. \cite[Theorem 1]{cg:fitting} says in particular that this Fitting ideal is equal to that of $\class(\ef_n^+) \teno{\Z} \Zprm$. Combining the above and passing to limits completes the proof.
\end{proof}

We observe the importance here of taking leading coefficients of $L$-functions at $s = 0$ rather than just values. For $\tsc$ even (\ie $\tsc^*$ odd), $\pi_n(\tsc^*)(\cycJ)$ concerns $L$-functions which are non-zero at $0$, and we get the usual Stickelberger elements which are related to \emph{minus} parts of class-groups via $\prm$-adic $L$-functions. However when $\tsc$ is odd (\ie $\tsc^*$ is even), $\pi_n(\tsc^*)(\cycJ)$ is concerned with $L$-functions having simple zeroes at $0$, which are related to \emph{plus} parts of class-groups via cyclotomic units.

\section{ ${\cal J}_{\ef/\gf}^{r}$ and annihilation} \label{5.4}

Let $\prm$ be an odd prime. Given $\alpha \in {\cal J}_{\ef/\gf}^{r}$ and $H \subseteq G = {\rm Gal}(\ef/\gf)$, choose any
\[   \beta \in   {\rm ann}_{{\mathbb Z}_{\prm}[H\ab]}( {\rm Tors} H_{{\rm \acute{e}t}}^{1}(
{\rm Spec}({\cal O}_{E^{[H,H]},S}) , {\mathbb Z}_{\prm}(1-r)))  .   \]
Then the $H$-component $B_{G}^{*}(\alpha)_{H}$ lies in ${\mathbb Q}_{\prm}[H\ab]^{N_{G}H}$, the fixed points under the conjugation action by $N_{G}H$, the normalizer of $H$ in $G$. Assuming Conjecture \ref{5.2}(i), $   B_{G}^{*}(\alpha)_{H} \cdot  \beta \in {\mathbb Z}_{\prm}[H\ab]^{N_{G}H}$. Choose $z_{H, \alpha, \beta}   \in {\mathbb Z}_{\prm}[H]$ such that
\[  \pi( z_{H, \alpha, \beta} ) =   B_{G}^{*}(\alpha)_{H} \cdot  \beta .  \]

Consider the composition
\[ \begin{array}{l}
H_{{\rm \acute{e}t}}^{2}({\rm Spec}({\cal O}_{E,S(E)}) , {\mathbb Z}_{\prm}(1-r)) \stackrel{ {\rm Tr}_{E/E^{[H,H]} }}{\ra} H_{{\rm \acute{e}t}}^{2}(
{\rm Spec}({\cal O}_{E^{[H,H]},S}) , {\mathbb Z}_{\prm}(1-r))   \\
\hspace{40pt}   \stackrel{ B_{G}^{*}(\alpha)_{H} \cdot  \beta }{\ra}
H_{{\rm \acute{e}t}}^{2}(
{\rm Spec}({\cal O}_{E^{[H,H]},S}) , {\mathbb Z}_{\prm}(1-r))  \\
\hspace{80pt}   \stackrel{ j }{\ra}
H_{{\rm \acute{e}t}}^{2}(
{\rm Spec}({\cal O}_{E,S(E)}) , {\mathbb Z}_{\prm}(1-r))
\end{array} \]
in which $j$ is induced by the inclusion of fields and ${\rm Tr}_{E/E^{[H,H]} }$ denotes the transfer homomorphism.

Assuming Conjecture \ref{5.2}(ii), this composition is zero. However, by Frobenius reciprocity for the cohomology transfer, for all $a \in H_{{\rm \acute{e}t}}^{2}({\rm Spec}({\cal O}_{E,S(E)}) , {\mathbb Z}_{\prm}(1-r))$
\[  \begin{array}{ll}
0    &  = j( \pi(z_{H, \alpha, \beta}) {\rm Tr}_{E/E^{[H,H]} }(a) ) \\
&  =  j \cdot {\rm Tr}_{E/E^{[H,H]} } ( z_{H, \alpha, \beta} \cdot a)  \\
&  =  (\sum_{h \in  {\rm Gal}(E/E^{[H,H]} )}  \  h  ) z_{H, \alpha, \beta} \cdot a  .
\end{array} \]

\begin{definition} \label{5.5}
In the situation of \S\ref{5.1} and \S\ref{5.4}, let $\ncJ{\ef/\gf,r} \subseteq {\mathbb Z}_{\prm}[G]$ denote the left ideal generated by the elements $ (\sum_{h \in  {\rm Gal}(E/E^{[H,H]} )}  \  h  ) z_{H, \alpha, \beta}$ as $\alpha$, $H$ and $\beta$ vary through all the possibilities above.
\end{definition}

\begin{theorem} \label{5.6}
If Conjecture \ref{5.2} is true for all abelian intermediate extensions $E^{[H,H]}/E^{H}$ of $\ef/\gf$ then  the left action of the left ideal $\ncJ{\ef/\gf,r}$ annihilates
\[    H_{{\rm \acute{e}t}}^{2}(
{\rm Spec}({\cal O}_{E,S(E)}) , {\mathbb Z}_{\prm}(1-r))  .   \]
\end{theorem}

\begin{remark} \label{5.7}
If $G$ is abelian in Definition \ref{5.5} and Theorem \ref{5.6}, then
\[   \ncJ{\ef/\gf,r} =    {\cal J}_{\ef/\gf}^{r} \cdot    {\rm ann}_{{\mathbb Z}_{\prm}[G]}( {\rm Tors} H_{{\rm \acute{e}t}}^{1}(
{\rm Spec}({\cal O}_{E, S(E)}) , {\mathbb Z}_{\prm}(1-r))) . \]
That is, $  \ncJ{\ef/\gf,r}$ equals the left hand side of Conjecture \ref{5.2}(ii).
\end{remark}

\begin{prop} \label{intJ is two-sided}
In Definition \ref{5.5}, $\ncJ{\ef/\gf,r}$ is a two-sided ideal in $\Zcop[G]$.
\end{prop}

\begin{proof}
In the notation of \S\ref{5.4}, it suffices to show that
\[ w \left(\sum_{h \in \gal{E/E^{\comm{H}}}} h\right) z_{H,\alpha,\beta} w^{-1} \]
lies in $\ncJ{\ef/\gf,r}$. Consider
\[ w \left(\sum_{h \in \gal{E,E^{\comm{H}}}} h \right) w^{-1} = \sum_{h \in \gal{E/E^{\comm{wHw^{-1}}}}} h \]
and $w z_{H,\alpha,\beta} w^{-1}$. Since $z_{H,\alpha,\beta}$ lies in $\Zcop[H]$ and maps to $B_G^*(\alpha) \beta$ in $\Zcop[H\ab]$, we see that $wz_{H,\alpha,\beta} w^{-1}$ lies in $\Zcop[wHw^{-1}]$ and maps to $wB_G^*(\alpha)_H w^{-1} w \beta w^{-1}$ in $\Zcop[H\ab]$. However, $w B_G^*(\alpha)_H w^{-1} = B_G^*(\alpha)_{wHw^{-1}}$ and $w\beta w^{-1}$ lies in
\[ \ann_{\Zcop[(wHw^{-1})\ab]}(\tors \het^1(\spec(\oh_{E^{\comm{wHw^{-1}}},S}),\Zcop(1-r))) ,\]
completing the proof.
\end{proof}

\begin{prop} \label{nc quotient}
Suppose that $F \con K \con E$ is a tower of number fields with $\ef/\gf$ and $K/F$ Galois. Then for $r = -1,-2,-3,\ldots$, the canonical homomorphism $\pi_{E/K} : \Zcop[\gal{\ef/\gf}] \ra \Zcop[\gal{K/F}]$ satisfies
\[ \pi_{E/K}(\ncJ{\ef/\gf,r}) \con \ncJ{K/F,r} .\]
\end{prop}

\begin{proof}
This follows easily from Lemma \ref{3.2} and Proposition \ref{4.4}.
\end{proof}

\section{Relation to Iwasawa theory} \label{sec iwasawa theory}

As discussed in the Introduction, the motivation for examining the behaviour of the fractional Galois ideal under changes of extension is to set up investigating a possible role in Iwasawa theory. Via the relationship of the fractional ideal with Stark-type elements (\eg cyclotomic units in the case $\lint = 0$ and Beilinson elements in the case $\lint < 0$, discussed in \cite{buckingham:frac} and \cite{snaith:rel} resp.), one might hope that an approach involving Euler systems would be fruitful here. A general connection of the fractional Galois ideal to Stark elements of arbitrary rank was demonstrated in \cite{buckingham:phd}, and the link of Stark elements with class-groups using the theory of Euler systems is discussed in \cite{rubin:kolyvagin,popescu:rubin}, so that a strategy as above would seem promising.

We conclude the paper with some speculation concerning what the non-commutative Iwasawa theory of Fukaya--Kato \cite{fk:noncomm}, Kato \cite{kato:heisenberg} and Ritter--Weiss \cite{rw:teitseries} suggests about $\hJ{\lint}{\ef/\gf}$ of Definition \ref{4.5} and $\ncJ{\ef/\gf,\lint}$ of Definition \ref{5.5}.

It is worth pointing out, before we begin the recapitulation proper, that \cite{fk:noncomm,kato:heisenberg,rw:teitseries} often restrict to the situation where the extension fields are totally real, which tends to involve only one of the eigenspaces of complex conjugation acting on $\hJ{\lint}{\ef/\gf}$ and $\ncJ{\ef/\gf,\lint}$. We have tried to give some examples (for example, \S \ref{commutative example}) which illustrate the expected role and properties of the other eigenspace.

Further, in this area there is an immense litany of conjectures (see \cite{fk:noncomm,burns:leading}) of which Stark's conjecture is approximately the weakest. All the constructions we have made are contingent \emph{only} on the truth of Stark's conjecture, which is crucial for us but also seems fundamental; it is assumed, for example, in \cite{rw:lrnc}.

Following \cite{kato:heisenberg}, let $\prm$ be an odd prime (denoted $p$ there), $F$ a totally real number field and $F_\infty$ a totally real Lie extension of $F$ containing $\Q(\zeta_{\prm^\infty})^+$. Here, $\Q(\zeta_{\prm^\infty})^+$ is the union of the totally real fields $\Q(\zeta_{\prm^n})^+ = \Q(\zeta_{\prm^n} + \zeta_{\prm^n}^{-1})$ over all $n \geq 1$. Let $G = \gal{F_\infty/F}$, and assume that only finitely many primes of $F$ ramify in $F_\infty$. Fix a finite set $\Sigma$ of primes of $F$ containing the ones which ramify in $F_\infty/F$. Define $\Lambda(G)$ to be the Iwasawa algebra of $G$, given by $\Lambda(G) = \Zprm\pwr{G} = \proo{U} \Zprm[G/U]$, where the limit runs over all open normal subgroups of $G$.

Let $C$ denote the cochain complex of $\Lambda(G)$-modules given by
\[ {\rm RHom}( \rm{R} \Gamma_{\acute{e}t}( {\cal O}_{F_{\infty}}[1/ \Sigma] , {\mathbb
Q}_{\prm}/{\mathbb Z}_{\prm}) , {\mathbb Q}_{\prm}/{\mathbb Z}_{\prm}) ,\]
so that $H^{0}(C) =
{\mathbb Z}_{\prm}$ with trivial $G$-action and $H^{-1}(C) = {\rm Gal}(M/
F_{\infty})$, the Galois group of the maximal pro-$\prm$ abelian extension of
$F_{\infty}$ unramified  outside $\Sigma$. The other $H^{i}(C)$'s are zero and
$ {\rm Gal}(M/ F_{\infty})$ is a finitely generated torsion (left)
$\Lambda(G)$-module. Let $F^{cyc} \subseteq F_{\infty}$ denote the cyclotomic
${\mathbb Z}_{\prm}$-extension and set $H = {\rm Gal}(  F_{\infty} / F^{cyc})
\subseteq G$ so that $G/H \cong {\mathbb Z}_{\prm}$. As in \cite{cfksv:main}, let
\[ S = \{f \in \Lambda(G) \sat \textrm{$\Lambda(G)/\Lambda(G)f$ is finitely generated as a $\Lambda(H)$-module}\} .\]
Then $S$ is an Ore set, which means that its elements may be inverted to form
the localized ring
$\Lambda(G)_{S}$, and there is an exact localization sequence of algebraic
K-groups
\[   K_{1}(\Lambda(G))  \ra   K_{1}(\Lambda(G)_{S})
\stackrel{\partial}{\ra}   K_{0}(\Lambda(G) ,  \Lambda(G)_{S} )
\ra   K_{0}(\Lambda(G))  \ra
 K_{0}(\Lambda(G)_{S})     .   \]
By \cite{hs:pseudonullity}, Iwasawa's conjecture concerning the vanishing of the
$\mu$-invariant implies that the cohomology of the perfect complex $C$ vanishes
when $S$-localized. This gives rise to a class $[C] \in   K_{0}(\Lambda(G) ,
\Lambda(G)_{S} )$. In the case of finite Galois extensions the class
$[C] $ accounts for the Stickelberger phenomena (c.f. \cite{snaith:stark}) but on the
other hand so do values of Artin $L$-functions. The main conjecture of
non-commutative Iwasawa theory, described below following
\cite{kato:heisenberg}, makes this relation clear in terms of
$\Lambda(G)_{S}$-modules.

There is an $\prm$-adic determinantal valuation which assigns to  $f \in
K_{1}(\Lambda(G)_{S}) $
and a continuous Artin representation $\rho$ a value $f(\rho) \in
\overline{{\mathbb Q}}_\prm  \cup \{\infty\} $. The main conjecture of
non-commutative Iwasawa theory asserts that there exists
 $\xi \in  K_{1}(\Lambda(G)_{S}) $ such that (i) $\partial(\xi) = -[C]$ and (ii)
$\xi(\rho \kappa^{r}) =
 L_{\Sigma}( 1-r , \rho)$ for any even $r \geq 2$ where $\kappa$ is the $\prm$-adic
cyclotomic character and $L_{\Sigma}(s, \rho)$ is the Artin $L$-function of
$\rho$ with the Euler factors at  $\Sigma$ removed.

The main conjecture of Iwasawa theory was formulated in \cite{rw:lrnc} and studied in the series of papers \cite{rw:teitseries} when the Lie group $G$ has rank zero or
one. The case of $G = GL_{2}(\mathbb Z_\prm)$ is of particular interest in the
study of elliptic curves $E/{\mathbb Q}$ without complex multiplication \cite{cfksv:main} and is proven for the $\prm$-adic Heisenberg group in \cite{kato:heisenberg}. For a comprehensive survey see \cite{fk:noncomm}.

Motivated by the main conjecture of Iwasawa theory, and more generally by the
role of $\Lambda(G)$ in the arithmetic geometry of elliptic curves and their
Selmer groups, there has been considerable ring-theoretic activity concerning
$\Lambda(G)$ and $\Omega(G) = \Lambda(G)/\prm \Lambda(G)$ (see \cite{ab:ringtheoretic,ab:primeness,awz:reflexive,venjakob:structuretheory,venjakob:weierstrass,venjakob:padiclie}). The
rings $\Lambda(G)$ and
$\Omega(G)$ are examples of ``just-infinite rings'' which both satisfy the
Auslander--Gorenstein condition and are thus amenable to Lie theoretic analysis.

In the survey article \cite{ab:ringtheoretic}, a number of questions are posed. In particular the constructions of \S7 are directly related to \cite[Question G]{ab:ringtheoretic}: ``Is there a mechanism for constructing ideals of Iwasawa algebras which involves neither central elements nor closed normal subgroups?''

\begin{prop} \label{limit of ncJ exists}
If $F_{\infty}/F$ is any $\prm$-adic Lie extension of a number field $F$ with
Galois group $G$ then, under the assumption of \S \ref{5.4} for the finite intermediate
subextensions $E/F$ for $r=-1, -2, -3, \ldots$ we may define a two-sided ideal
\[ \ncJ{F_{\infty}/F , r} = \proo{E} \ncJ{E/F,r} \]
in $\Lambda(G)$, where the limit is taken over finite Galois subextensions $E/F$ of $F_{\infty}/F$.
\end{prop}

In view of the annihilation discussion of \S \ref{5.4}, Proposition \ref{limit of ncJ exists} suggests
the following:
\begin{question} \label{8.2}
What is the intersection of the canonical Ore set $S$ of \cite{cfksv:main,kato:heisenberg} with $\ncJ{F_{\infty}/F , r}$?
\end{question}

In many ways the most interesting case is when $G = GL_{2}(\mathbb Z_{\prm})$ ($\prm \geq 7$) arising from the tower of $\prm$-primary torsion points on an elliptic curve over ${\mathbb Q}$ without complex multiplication \cite{coates:fragments,cfksv:main}. In this case one has particularly strong information concerning two-sided primes ideals of $\Lambda(G)$ -- see \cite{awz:reflexive}. There is a possibly alternative approach to the construction of fractional Galois ideals in ${\mathbb Q}_{\prm}[{\rm Gal}(K/ {\mathbb Q})]$ based on assuming that a type of Stark conjecture holds for the Hasse--Weil $L$-function of the elliptic curve \cite{stopple:stark}. It would be interesting to know whether this leads to the same two-sided ideal as Proposition \ref{limit of ncJ exists}.


\begin{thebibliography}{10}

\bibitem{ab:ringtheoretic}
K.~Ardakov and K.~A. Brown.
\newblock Ring-theoretic properties of {I}wasawa algebras: a survey.
\newblock {\em Doc. Math.}, (Extra Vol.):7--33 (electronic), 2006.

\bibitem{ab:primeness}
K.~Ardakov and K.~A. Brown.
\newblock Primeness, semiprimeness and localisation in {I}wasawa algebras.
\newblock {\em Trans. Amer. Math. Soc.}, 359(4):1499--1515 (electronic), 2007.

\bibitem{awz:reflexive}
K.~Ardakov, F.~Wei, and J.~J. Zhang.
\newblock Reflexive ideals in {I}wasawa algebras.
\newblock {\em To appear in Advances in Mathematics}.

\bibitem{brumer:units}
Armand Brumer.
\newblock On the units of algebraic number fields.
\newblock {\em Mathematika}, 14:121--124, 1967.

\bibitem{buckingham:frac}
Paul Buckingham.
\newblock The canonical fractional {G}alois ideal at $s = 0$.
\newblock {\em To appear in the Journal of Number Theory}.

\bibitem{buckingham:phd}
Paul Buckingham.
\newblock {P}h{D} thesis, {U}niversity of {S}heffield.
\newblock 2008.

\bibitem{burns:leading}
D.~Burns.
\newblock Leading terms and values of equivariant motivic {$L$}-functions.
\newblock {\em To appear in the special volume of the Quarterly Journal of Pure
  and Applied Mathematics on the occasion of Tate's 80th birthday}.

\bibitem{cassou-nogues:valeurs}
Pierrette Cassou-Nogu{\`e}s.
\newblock Valeurs aux entiers n\'egatifs des fonctions z\^eta et fonctions
  z\^eta {$p$}-adiques.
\newblock {\em Invent. Math.}, 51(1):29--59, 1979.

\bibitem{cs:padic}
J.~Coates and W.~Sinnott.
\newblock On {$p$}-adic {$L$}-functions over real quadratic fields.
\newblock {\em Invent. Math.}, 25:253--279, 1974.

\bibitem{coates:fragments}
John Coates.
\newblock Fragments of the {${\rm GL}\sb 2$} {I}wasawa theory of elliptic
  curves without complex multiplication.
\newblock In {\em Arithmetic theory of elliptic curves (Cetraro, 1997)}, volume
  1716 of {\em Lecture Notes in Math.}, pages 1--50. Springer, Berlin, 1999.

\bibitem{cfksv:main}
John Coates, Takako Fukaya, Kazuya Kato, Ramdorai Sujatha, and Otmar Venjakob.
\newblock The {$\rm GL\sb 2$} main conjecture for elliptic curves without
  complex multiplication.
\newblock {\em Publ. Math. Inst. Hautes \'Etudes Sci.}, (101):163--208, 2005.

\bibitem{cs:stickel}
John Coates and Warren Sinnott.
\newblock An analogue of {S}tickelberger's theorem for the higher {$K$}-groups.
\newblock {\em Invent. Math.}, 24:149--161, 1974.

\bibitem{cg:fitting}
Pietro Cornacchia and Cornelius Greither.
\newblock Fitting ideals of class groups of real fields with prime power
  conductor.
\newblock {\em J. Number Theory}, 73(2):459--471, 1998.

\bibitem{dr:abelianlfunctions}
Pierre Deligne and Kenneth~A. Ribet.
\newblock Values of abelian {$L$}-functions at negative integers over totally
  real fields.
\newblock {\em Invent. Math.}, 59(3):227--286, 1980.

\bibitem{fk:noncomm}
Takako Fukaya and Kazuya Kato.
\newblock A formulation of conjectures on {$p$}-adic zeta functions in
  noncommutative {I}wasawa theory.
\newblock In {\em Proceedings of the St. Petersburg Mathematical Society. Vol.
  XII}, volume 219 of {\em Amer. Math. Soc. Transl. Ser. 2}, pages 1--85,
  Providence, RI, 2006. Amer. Math. Soc.

\bibitem{gross:higherstark}
B.~H. Gross.
\newblock On the values of {A}rtin {$L$}-functions.
\newblock {\em Unpublished preprint}.

\bibitem{hs:pseudonullity}
Yoshitaka Hachimori and Romyar~T. Sharifi.
\newblock On the failure of pseudo-nullity of {I}wasawa modules.
\newblock {\em J. Algebraic Geom.}, 14(3):567--591, 2005.

\bibitem{hayes:nonabstick}
David~R. Hayes.
\newblock Stickelberger functions for non-abelian {G}alois extensions of global
  fields.
\newblock In {\em Stark's conjectures: recent work and new directions}, volume
  358 of {\em Contemp. Math.}, pages 193--206. Amer. Math. Soc., Providence,
  RI, 2004.

\bibitem{iwasawa:padicl}
Kenkichi Iwasawa.
\newblock {\em Lectures on {$p$}-adic {$L$}-functions}.
\newblock Princeton University Press, Princeton, N.J., 1972.
\newblock Annals of Mathematics Studies, No. 74.

\bibitem{kato:heisenberg}
Kazuya Kato.
\newblock Iwasawa theory of totally real fields for {G}alois extensions of
  {H}eisenberg type.
\newblock {\em Preprint}.

\bibitem{kl:padicl}
Tomio Kubota and Heinrich-Wolfgang Leopoldt.
\newblock Eine {$p$}-adische {T}heorie der {Z}etawerte. {I}. {E}inf\"uhrung der
  {$p$}-adischen {D}irichletschen {$L$}-{F}unktionen.
\newblock {\em J. Reine Angew. Math.}, 214/215:328--339, 1964.

\bibitem{lang:algebra2nd}
Serge Lang.
\newblock {\em Algebra}.
\newblock Addison-Wesley Publishing Company Advanced Book Program, Reading, MA,
  second edition, 1984.

\bibitem{martinet:lfunctions}
J.~Martinet.
\newblock Character theory and {A}rtin {$L$}-functions.
\newblock In {\em Algebraic number fields: $L$-functions and Galois properties
  (Proc. Sympos., Univ. Durham, Durham, 1975)}, pages 1--87. Academic Press,
  London, 1977.

\bibitem{mw:classfields}
B.~Mazur and A.~Wiles.
\newblock Class fields of abelian extensions of {${\bf Q}$}.
\newblock {\em Invent. Math.}, 76(2):179--330, 1984.

\bibitem{popescu:rubin}
Cristian~D. Popescu.
\newblock Rubin's integral refinement of the abelian {S}tark conjecture.
\newblock In {\em Stark's conjectures: recent work and new directions}, volume
  358 of {\em Contemp. Math.}, pages 1--35. Amer. Math. Soc., Providence, RI,
  2004.

\bibitem{quillen:ktheory}
Daniel Quillen.
\newblock Higher algebraic {$K$}-theory. {I}.
\newblock In {\em Algebraic $K$-theory, I: Higher $K$-theories (Proc. Conf.,
  Battelle Memorial Inst., Seattle, Wash., 1972)}, pages 85--147. Lecture Notes
  in Math., Vol. 341. Springer, Berlin, 1973.

\bibitem{rw:teitseries}
J\"urgen Ritter and Alfred Weiss.
\newblock Toward equivariant {I}wasawa theory.
\newblock Part I, \emph{Manuscripta Math.}, 109(2):131--146, 2002; Part II,
  \emph{Indag. Math. (N.S.)}, 15(4):549--572, 2004; Part III, \emph{Math.
  Ann.}, 336(1):27--49, 2006; Part IV, \emph{Homology, Homotopy Appl.},
  7(3):155--171, 2005 (electronic).

\bibitem{rw:lrnc}
J{\"u}rgen Ritter and Alfred Weiss.
\newblock The lifted root number conjecture and {I}wasawa theory.
\newblock {\em Mem. Amer. Math. Soc.}, 157(748):viii+90, 2002.

\bibitem{rubin:kolyvagin}
Karl Rubin.
\newblock Stark units and {K}olyvagin's ``{E}uler systems''.
\newblock {\em J. Reine Angew. Math.}, 425:141--154, 1992.

\bibitem{serre:linear}
Jean-Pierre Serre.
\newblock {\em Linear representations of finite groups}.
\newblock Springer-Verlag, New York, 1977.
\newblock Translated from the second French edition by Leonard L. Scott,
  Graduate Texts in Mathematics, Vol. 42.

\bibitem{snaith:equiv}
Victor Snaith.
\newblock Equivariant motivic phenomena.
\newblock In {\em Axiomatic, enriched and motivic homotopy theory}, volume 131
  of {\em NATO Sci. Ser. II Math. Phys. Chem.}, pages 335--383. Kluwer Acad.
  Publ., Dordrecht, 2004.

\bibitem{snaith:ebi}
Victor~P. Snaith.
\newblock {\em Explicit {B}rauer induction}, volume~40 of {\em Cambridge
  Studies in Advanced Mathematics}.
\newblock Cambridge University Press, Cambridge, 1994.
\newblock With applications to algebra and number theory.

\bibitem{snaith:rel}
Victor~P. Snaith.
\newblock Relative {$K\sb 0$}, annihilators, {F}itting ideals and
  {S}tickelberger phenomena.
\newblock {\em Proc. London Math. Soc. (3)}, 90(3):545--590, 2005.

\bibitem{snaith:stark}
Victor~P. Snaith.
\newblock Stark's conjecture and new {S}tickelberger phenomena.
\newblock {\em Canad. J. Math.}, 58(2):419--448, 2006.

\bibitem{stopple:stark}
Jeffrey Stopple.
\newblock Stark conjectures for {CM} elliptic curves over number fields.
\newblock {\em J. Number Theory}, 103(2):163--196, 2003.

\bibitem{tate:stark}
John Tate.
\newblock {\em Les conjectures de {S}tark sur les fonctions {$L$} d'{A}rtin en
  {$s=0$}}, volume~47 of {\em Progress in Mathematics}.
\newblock Birkh\"auser Boston Inc., Boston, MA, 1984.
\newblock Lecture notes edited by Dominique Bernardi and Norbert Schappacher.

\bibitem{venjakob:structuretheory}
Otmar Venjakob.
\newblock On the structure theory of the {I}wasawa algebra of a {$p$}-adic
  {L}ie group.
\newblock {\em J. Eur. Math. Soc. (JEMS)}, 4(3):271--311, 2002.

\bibitem{venjakob:weierstrass}
Otmar Venjakob.
\newblock A non-commutative {W}eierstrass preparation theorem and applications
  to {I}wasawa theory.
\newblock {\em J. Reine Angew. Math.}, 559:153--191, 2003.
\newblock With an appendix by Denis Vogel.

\bibitem{venjakob:padiclie}
Otmar Venjakob.
\newblock On the {I}wasawa theory of {$p$}-adic {L}ie extensions.
\newblock {\em Compositio Math.}, 138(1):1--54, 2003.

\bibitem{wash:cyc}
Lawrence~C. Washington.
\newblock {\em Introduction to cyclotomic fields}, volume~83 of {\em Graduate
  Texts in Mathematics}.
\newblock Springer-Verlag, New York, second edition, 1997.

\bibitem{wiles:iwasawa}
A.~Wiles.
\newblock The {I}wasawa conjecture for totally real fields.
\newblock {\em Ann. of Math. (2)}, 131(3):493--540, 1990.

\end{thebibliography}
\end{document}